\documentclass{article}
\usepackage[applemac]{inputenc}
\usepackage[T1]{fontenc}
\usepackage{lmodern}
\usepackage[a4paper]{geometry}
\usepackage[english]{babel}
\usepackage{graphicx}
\usepackage{onlyamsmath}
\usepackage{amsfonts}
\usepackage{amssymb}
\usepackage{ntheorem}

\newcommand{\e}{\epsilon}
\newcommand{\se}{\sqrt{\epsilon}}
\newcommand{\E}{\mathbb{E}}
\newcommand{\Pro}{\mathbb{P}}

\newcommand{\ga}{\gamma}

\newcommand{\xt}[1]{\mathbf{#1}}

\newtheorem{thm}{Theorem}[section]
\newtheorem{prop}{Proposition}[section]

\newtheorem{lem}{Lemma}[section]
\theoremstyle{nonumberplain}
\theorembodyfont{\upshape}
\newtheorem{proof}{Proof}

\title{Wave Decoherence for the Random Schrödinger Equation with Long-Range Correlations}

\author{Christophe Gomez\thanks{Department of Mathematics, Stanford University, Building 380, Sloan Hall
Stanford, California 94305 USA (chgomez@math.stanford.edu). Tel: +(1)650-723-1968. Fax: +(1)650-725-4066.}}

\begin{document}

\maketitle

\begin{abstract} 

In this paper, we study the loss of coherence of a wave propagating according to the Schrödinger equation with a time-dependent random potential. The random potential is assumed to have slowly decaying correlations. The main tool to analyze the decoherence phenomena is a properly rescaled Wigner transform of the solution of the random Schrödinger equation. We exhibit anomalous wave decoherence effects at different propagation scales.

\end{abstract}

\begin{flushleft}
\end{flushleft}

\begin{flushleft}
\end{flushleft}

\section{Introduction.}\label{intro}

The Schrödinger equation with a time-dependent random potential has attracted lots of attention because of its large domains of applications \cite{blomgren, erdos, erdos2, erdos3, ho,spohn}. It is widely used for instance in wave propagation under the paraxial or parabolic approximation \cite{bal3, bal0, bal,bal4, bal2}. This field of research was recently stimulated \cite{bal2,garnier, gomez,marty, marty2, solna} by data collections in wave propagation experiments showing that the medium of propagation presented some long-range effects \cite{dolan, sidi}. Most of the theoretical studies regarding wave propagation in long-range random media hold in one dimensional propagation media, which are very convenient for mathematical studies but not relevant in many applications. 

In this paper, we consider the following random Schrödinger equation 
\begin{equation}\label{eqintro}\begin{split}
i\partial_t \phi+\frac{1}{2}\Delta_{\xt{x}}\phi-\se V(t,\xt{x})\phi&=0, \quad t\geq 0\text{ and }\xt{x}\in \mathbb{R}^d,\\
\phi(0,\xt{x})=\phi_0(\xt{x}),
\end{split}\end{equation}
with a random potential $V(t,\xt{x})$, which is a spatially and temporally homogeneous mean-zero random field. Here, $t\geq0$ represents the temporal variable, $\xt{x}\in\mathbb{R}^d$ the spatial variable with $d\geq 1$, and $\e\ll1$ is a small parameter which represents the relative strength of the random fluctuations. A classical tool to study the \emph{loss of coherence} of the wave field $\phi$ is the Wigner transform defined by
\[ W_\e(t,\xt{x},\xt{k})=\frac{1}{(2\pi)^d}\int d\xt{y} e^{i \xt{k}\cdot\xt{y}}\phi\Big(t,\xt{x}-\frac{\xt{y}}{2}\Big)\overline{\phi\Big(t,\xt{x}+\frac{\xt{y}}{2}\Big)},\]
which is somehow the Fourier transform of the correlation function in space of $\phi$ at point $\xt{x}$. The Wigner transform measures the degree of correlation of the wave field $\phi$ at two different points ($\xt{x}-\xt{y}/2$ and $\xt{x}+\xt{y}/2$). The \emph{loss of coherence} of the wave field $\phi$ (or equivalently the \emph{wave decoherence}) corresponds to the evolution of the degree of correlation in space of the wave field $\phi$ and is captured by the wave number $\xt{k}$ through the evolution in time of the momentum of the Wigner transform $W$. Let us remark that if $V=0$ in \eqref{eqintro}, we have $W(t,\xt{x},\xt{k})=W_0(\xt{x}-t\xt{k},\xt{k})$, where $W_0$ is the Wigner transform of the initial data $\phi_0$. In this case the momentum of $W$ is preserved during the propagation, that is there is no variation of the momentum with respect to time meaning that there is no loss of decoherence.

We refer to \cite{gerard,lions} for the basic properties of the Wigner transform. In our problem the amplitude of the random perturbations are small, so that to observe significant cumulative stochastic effects we have to look at the wave field $\phi$ over long times and large propagation distances. Consequently, we consider the rescaled field 
\begin{equation}\label{phis1}\phi_\e(t,\xt{x})=\phi\Big(\frac{t}{\e^s},\frac{\xt{x}}{\e^s}\Big)\end{equation}
satisfying the scaled random Schrödinger equation 
\[\begin{split}
i\e^s\partial_t \phi_\e+\frac{\e^{2s}}{2}\Delta_{\xt{x}}\phi_\e-\se V\Big(\frac{t}{\e^{s}},\frac{\xt{x}}{\e^s}\Big)\phi_\e&=0, \quad t\geq 0\text{ and }\xt{x}\in \mathbb{R}^d,\\
\phi_\e(0,\xt{x})=\phi_{0,\e}(\xt{x}),
\end{split}\]
and where  $\boldsymbol{s\in(0,1]}$ \textbf{is the propagation scale parameter}.  If the random potential $V$ has rapidly decaying correlations, it has been shown \cite{bal2} that the wave field $\phi_\e$ does not exhibit any significant random perturbations before the propagation scale $\e^{-1}$ ($s=1$). More precisely, the Fourier transform in space of $\phi_\e$ (defined by \eqref{phis1}) properly scaled converges point-wise to a stochastic complex Gaussian limit with a one-time statistic of an Ornstein-Uhlenbeck process. Wave decoherence phenomena have been studied in this context in many papers \cite{bal0, bal, erdos, erdos2, erdos3, fannjiang3, fannjiang2, ho, lukkarinen, spohn} thanks to the following Wigner transform for $s=1$
\begin{equation}\label{wignerint} \begin{split}
W_\e(t,\xt{x},\xt{k})&=\frac{1}{(2\pi)^d}\int d\xt{y} e^{i \xt{k}\cdot\xt{y}}\phi_\e\Big(t,\xt{x}-\e\frac{\xt{y}}{2}\Big)\overline{\phi_\e\Big(t,\xt{x}+\e\frac{\xt{y}}{2}\Big)}\\
&=\frac{1}{(2\pi)^d}\int d\xt{y} e^{i \xt{k}\cdot\xt{y}}\phi\Big(\frac{t}{\e},\frac{\xt{x}}{\e}-\frac{\xt{y}}{2}\Big)\overline{\phi\Big(\frac{t}{\e},\frac{\xt{x}}{\e}+\frac{\xt{y}}{2}\Big)}.
\end{split}\end{equation}
The Wigner transform \eqref{wignerint} analyzes the loss of coherence of $\phi_\e$ on a spatial correlation scale of order  $\e$ at the macroscopic propagation scale $\e^{-1}$, and the loss of coherence of $\phi$ on a spatial correlation scale of order  $1$ at the microscopic propagation scale. Let us note that 
\[\lvert \phi_\e(t,x)\rvert^2=\int d\xt{k} W_\e(t,\xt{x},\xt{k}),\]
so that we may think of $W_\e$ as a wave number resolved energy density. However, the terminology density is not well appropriate since $W_\e$ is not always positive but it becomes positive in the limit $\e\to 0$ \cite{lions}. 

In several context involving rapidly decorrelating random potentials, it has been shown that the expectation of the Wigner transform $\E[W_\e(t,\xt{x},\xt{k})]$ converges as $\e$ goes to $0$ to the solution $W$ of the radiative transport equation 
\begin{equation}\label{radtransint}
\partial_t W +\xt{k}\cdot\nabla_\xt{x} W= \int d\xt{p}\sigma(\xt{p},\xt{k})(W(t,\xt{x},\xt{p})-W(t,\xt{x},\xt{k})),
\end{equation}
where the transfer coefficient $\sigma(\xt{p},\xt{k})$ depends on the power spectrum of the two-point correlation function of the random potential $V$. Moreover, in some cases \cite{bal3, bal,bal4, fannjiang3}, it has been shown that the limit $W$ is often self-averaging, that is $W_\e$ converges in probability to the deterministic limit $W$ for the weak topology on $L^2(\mathbb{R}^{2d})$. In other words, the wave decoherence mechanism described by the radiative transfer equation \eqref{radtransint} does not depend on the particular realization of the random medium.

In this paper, we investigate the loss of coherence of a wave propagating according to the Schrödinger equation \eqref{eqintro} involving a random potential $V$ with slowly decaying correlations defined in Section \ref{grfield}. In this context, the behavior of the scaled field $\phi_\e$ defined by \eqref{phis1} evolves on different propagation scales $\e^{-s}$ \cite{bal2, gomez}. In \cite[Theorem 1.2]{bal2} the authors study $\phi_\e$ itself on the propagation scales $\e^{-s}$, with $s=1/(2\kappa)$ and $\kappa>1/2$, and show that the Fourier transform of $\phi_\e$ properly scaled converges point-wise in distribution to a complex exponential of a fractional Brownian motion with Hurst index $\kappa$, where $\kappa$ depends on the statistic of the random potential (see Section \ref{sectionsk}). This result means that $\e^{-s}$, with $s=1/(2\kappa)<1$, is the first propagation scale on which the random perturbations become significant, and induces a random phase modulation on the propagating wave. In \cite[Theorem 2.2]{gomez} the author studies the loss of coherence of $\phi_\e$ on the propagation scale $\e^{-1}$ ($s=1$), and shows that the Wigner transform \eqref{wignerint} of $\phi_\e$ converges in probability for the weak topology on $L^2(\mathbb{R}^{2d})$ to the unique solution of a deterministic radiative transfer equation similar to \eqref{radtransint}. In other words, the wave decoherence happening on this propagation scale ($\e^{-1}$) does not depend on the particular realization of the random potential. However, even if the radiative transfer equation has similar structure in both cases (for rapidly and slowly decaying correlations), the presence of a potential with long-range correlations have a striking effect. In contrast with the rapidly decorrelating case studied in \cite{bal0}, the scattering coefficient $\Sigma(\xt{k})=\int d\xt{p} \sigma(\xt{k},\xt{p})=+\infty$ is not defined anymore. The radiative transfer equation is however still well defined because of the difference $W(t,\xt{x},\xt{p})-W(t,\xt{x},\xt{k})$ which balances the singularity introduced by the long-range correlation assumption.  Moreover, these long-range correlations imply an instantaneous regularizing effect of the radiative transfer equation \cite[Theorem 3.1]{gomez}. Consequently, these results show a qualitative and thorough difference between the rapidly and slowly deccorelating cases. In fact, in contrast with the rapidly decorrelating case, for which the phase and the phase space density evolve on the same propagation scale $\e^{-1}$ \cite{bal2}, the phase of $\phi_\e$ and its phase space energy now evolve on different propagation scales. 

The main goal of this paper is to study the loss of coherence of a wave propagating according to the Schrödinger equation \eqref{eqintro} involving a random potential $V$ with slowly decaying correlations for propagation scale parameters $s\in(1/(2\kappa),1)$. In fact, for $s>1/(2\kappa)$ the random phase modulation obtained for $s=1/(2\kappa)$ implies that the wave field $\phi_\e$ produces very fast phase modulation, so that for sufficiently long times and large propagation distances  $\e^{-s}$ (with $s>1/(2\kappa)$) the wave coherence should be broken. Let us note that for a given propagation scale parameter $s$, the wave decoherence can be too small to be observed. For instance in \cite{gomez}, the author shows using the Wigner transform \eqref{wignerint} that there is no significant wave decoherence on the wave field $\phi_\e$ on spatial scales of order $\e^s$, before $s=1$. As we will see in Section \ref{sectionint}, for $s<1$ wave decoherence takes place first on large spatial scales and then propagates to the smaller ones as the propagation scale parameter $s$ increases. The larger the propagation scale parameter $s$ is the smaller the spatial scale is to observe wave decoherence (see Figure \ref{schema}).   To exhibit wave decoherence for $s<1$, we need to consider a properly scaled Wigner transform of the field $\phi_\e$ (see Section \ref{wignersection}). Depending on the propagation scale parameter $s$, we show that this scaled Wigner transform converges in probability, for the weak topology on $L^2(\mathbb{R}^{2d})$ to the unique solution of a fractional diffusion equation. This momentum diffusion equation describes the wave decoherence mechanism, and shows that it does not depend on the particular realization of the random potential. The anomalous momentum diffusions obtained for $s\in(1/(2\kappa),1]$ allow us to exhibit a damping coefficient, describing the decoherence rate, obeying a power law with exponent in $(0,1)$. 

The organization of this paper is as follows: In Section \ref{section1}, we present the random Schrödinger equation that will be studied in this paper; then we present the construction of the random potential and we introduce the long-range correlation assumption on the potential $V$ used throughout this paper; finally, we introduce the rescale Wigner transform which is the main tool to describe the wave decoherence mechanisms. The results stated in Section \ref{sectionsk} and Section \ref{sections1} have been respectively shown in \cite{bal2} and \cite{gomez}, but we recall these results to provide a self-contained presentation of the wave decoherence phenomena. In Section \ref{sectionsk}, we present the behavior of the field $\phi_\e$ for the propagation scale parameter $s=1/(2\kappa)$.  In Section \ref{sectionint}, we state the main result of this paper. We present the asymptotic behavior in long-range random media of a properly scaled Wigner transform over the intermediate range of propagation scale parameter $s\in(1/(2\kappa),1)$. In Section \ref{sections1}, we describe the asymptotic evolution in long-range random media of the phase space energy density of the solution of the random Schrödinger equation for $s=1$. Finally, in Section \ref{proofpropmar} we recall the proof of Proposition \ref{propmar}, and Section \ref{proof} is devoted to the proofs of Theorem \ref{thasymptotic} and Theorem \ref{thasymptotic3}.

\section*{Acknowledgment}

This work was supported by AFOSR FA9550-10-1-0194 Grant. I wish to thank Lenya Ryzhik for his suggestions, and the referee for his careful reading of the manuscript.

\section{The Random Schrödinger equation}\label{section1}

This section introduces in a first time the random Schrödinger equation studied in this paper, then in a second time it introduces the random potential with long-range correlation properties. Finally, we introduce the Wigner transform which is the main tool in this paper to study the loss of coherence of solutions of the Schrödinger equation.
 
We consider the dimensionless form of the Schrödinger equation on $\mathbb{R}^d$ with a time-dependent random potential:
\begin{equation}\label{schrodingereq0}
i \partial_t \phi+\frac{1}{2}\Delta_{\xt{x}} \phi-\e^{\frac{1-\ga}{2}} V\Big(\frac{t}{\e^{\gamma}},\xt{x}\Big)\phi =0,\end{equation}
with $\ga\in [0,1)$. $\ga$ is a parameter characterizing the correlation length in time. If $\ga=0$ the correlation lengths in space and time are of the same order, but if $\ga\in(0,1)$ the correlation length in time is small compared to the correlation length in space.  In \eqref{schrodingereq0} the strength of the random perturbations are small, so that to observe significant cumulative stochastic effects we have to look at the wave field $\phi$ over long times and large propagation distances. Consequently, throughout this paper we consider the following rescaled wave field :
\begin{equation}\label{phis}\phi_\e(t,\xt{x})=\phi\Big(\frac{t}{\e^s},\frac{\xt{x}}{\e^s}\Big), \quad \text{with}\quad s\in(0,1].\end{equation}
Moreover, throughout this paper the parameter $\boldsymbol{s\in(0,1]}$ \textbf{represents the propagation scale parameter}, and the scaled wave field $\phi_\e$ satisfies the scaled Schrödinger equation 
\begin{equation}\label{schrodingereq}
i\e^s \partial_t \phi_\e+\frac{\e^{2s}}{2}\Delta_{\xt{x}} \phi_\e-\e^{\frac{1-\ga}{2}} V\Big(\frac{t}{\e^{s+\gamma}},\frac{\xt{x}}{\e^s}\Big)\phi_\e =0\quad \text{with}\quad \phi_\e(0,\xt{x})=\phi_{0,\e}(\xt{x}).
\end{equation}
Here $\Delta_{\xt{x}}$ is the Laplacian on $\mathbb{R}^d$ given by $\Delta=\sum_{j=1}^d \partial_{x_j}^2$. $(V(t,\xt{x}),t\geq0,\xt{x}\in \mathbb{R}^d)$ is a random potential given by a  stationary zero-mean continuous random process in space and time, and whose properties are described in Section \ref{grfield}. The initial datum $\phi_{0,\e}(\xt{x})=\phi_{0,\e}(\xt{x},\zeta)$ is a random function with respect to a probability space $(S,\mathcal{S},\mu(d\zeta))$, and independent to the random potential $V$. This randomness on the initial data is called mixture of states. This terminology comes from the quantum mechanics, and the reason for introducing this additional randomness will be explained more precisely in Section \ref{wignersection}.

\subsection{Random potential}\label{grfield}

This section is devoted to the introduction of the random potential $V$ considered in this paper, and is also a short remainder about some properties of Gaussian random fields that we use in the proof of Theorem \ref{thasymptotic} and Theorem \ref{thasymptotic3}. All the properties of the random field $V$ exposed in this section result from the standard properties of Gaussian random fields presented in \cite{adlertaylor} for instance.

Before introducing the random potential we need some notations. Let $\widehat{R}_0$ be a positive function such that $\widehat{R}_0\in L^1(\mathbb{R}^d)$, $\widehat{R}_0(-\xt{p})=\widehat{R}_0(\xt{p})$, decaying rapidly at infinity and having a singularity at $\xt{p}=0$. Let us consider the following Hilbert space : 
\[\mathcal{H}_m=\Big\{\varphi \text{ such that}\quad\varphi(\xt{p})=\overline{\varphi(-\xt{p})}\quad\text{and}\quad \int_{\mathbb{R}^d}m(d\xt{p})\lvert \varphi(\xt{p})\rvert^2 <+\infty\Big\},\] 
equipped with the inner product
\[ \big<\varphi,\psi\big>_{\mathcal{H}_m}=\int_{\mathbb{R}^d} m(d\xt{p})\,\,\varphi (\xt{p}) \overline{\psi(\xt{p})}\quad \forall (\varphi,\psi)\in \mathcal{H}_m\times\mathcal{H}_m,\]
and where $m(d\xt{p})=\widehat{R}_0(\xt{p})d\xt{p}$ is a real finite measure. Let $(v_n)_{n\geq 0}$ be an orthonormal basis of $\mathcal{H}_m$ and let $(\widehat{V}_n)_{n\geq0}$ be a family of  real-valued stationary zero-mean Gaussian processes such that
\[\E[\widehat{V}_p(t)\widehat{V}_q(s)]=\int_{\mathbb{R}^d} m(d\xt{p}) e^{-\mathfrak{g}(\xt{p}) \lvert t-s\rvert } v_p(\xt{p})\overline{v_q(\xt{p})},\quad \forall s,t\geq 0\text{ and }\forall p,q\geq 0.\]
Now, let us consider the following real-valued linear functional on $\mathcal{H}_m$ defined by
\[ \widehat{V}(t)(\varphi)=\big<\widehat{V}(t),\varphi\big>_{\mathcal{H}'_m,\mathcal{H}_m}=\sum_{n\geq 0} \widehat{V}_n(t) \big<v_n,\varphi\big>_{\mathcal{H}_m},\]
where $\mathcal{H}'_m$ stands for the dual space of $\mathcal{H}_m$, which is well defined since
\[
\E[\lvert \big<\widehat{V}(t),\varphi\big>_{\mathcal{H}'_m,\mathcal{H}_m} \rvert^2]=\sum_{p,q\geq 0} \big<v_p,v_q\big>_{\mathcal{H}_m}  \big<v_p,\varphi\big>_{\mathcal{H}_m}\overline{ \big<v_q,\varphi\big>_{\mathcal{H}_m}}=\|\varphi\|^2_{\mathcal{H}_m}, \quad \forall \varphi\in\mathcal{H}_m.
\]
As a result, $(\widehat{V}(t))_{t\geq 0}$ is a real-valued stationary zero-mean Gaussian process on  $\mathcal{H}'_m$ such that
\[
\mathbb{E}\big[\big<\widehat{V}(t),\varphi\big>_{\mathcal{H}'_m,\mathcal{H}_m}\big<\widehat{V}(s),\psi\big>_{\mathcal{H}'_m,\mathcal{H}_m}\big]=\int_{\mathbb{R}^d}m(d\xt{p})\,\,e^{-\mathfrak{g}(\xt{p}) \lvert t-s\rvert }\varphi(\xt{p})\overline{\psi(\xt{p})},
\]
for all $t,s\geq 0$ and for all $(\varphi,\psi)\in\mathcal{H}_m\times\mathcal{H}_m$, which corresponds to a covariance function given by
\[\E[\widehat{V}(t,d\xt{p}_1)\widehat{V}(s,d\xt{p}_2)]=(2\pi)^d \tilde{R}(t-s,\xt{p}_1)\delta(\xt{p}_1+\xt{p}_2),\quad \]
where the spatial power spectrum is given by
\begin{equation}\label{rtilde}
\tilde{R}(t,\xt{p})=e^{-\mathfrak{g}(\xt{p})\lvert t\rvert }\widehat{R}_0(\xt{p}).
\end{equation} 
Here, the nonnegative function $\mathfrak{g}$, such that $\mathfrak{g}(\xt{p})=\mathfrak{g}(-\xt{p})$, is the spectral gap. Particular assumptions involving the spectral gap $\mathfrak{g}$ will be introduced at the end of this section to ensure long-range correlation properties on the potential $V$ in \eqref{schrodingereq} which is defined as follows :
\begin{equation}\label{vtx} V(t,\xt{x})=\frac{1}{(2\pi)^d}\int \widehat{V}(t,d\xt{p})e^{i\xt{p}\cdot\xt{x}},\quad \forall t\geq 0\text{ and }\forall\xt{x}\in\mathbb{R}^d,\end{equation}
so that $(V(t,\xt{x}),t\geq0,\xt{x}\in \mathbb{R}^d)$ is a real-valued stationary zero-mean real Gaussian process with
\begin{equation}\label{covfunc}\begin{split} \mathbb{E}\big[V(t,\xt{x})V(s,\xt{y})\big]=R(t-s,\xt{x}-\xt{y})&=
\frac{1}{(2\pi)^{d}}\int d\xt{p}\tilde{R}(t-s,\xt{p})e^{i\xt{p}\cdot(\xt{x}-\xt{y})}\\
&=\frac{1}{(2\pi)^{d+1}}\int d\omega d\xt{p}\widehat{R}(\omega,\xt{p})e^{i\omega(t-s)}e^{i\xt{p}\cdot(\xt{x}-\xt{y})},\end{split}\end{equation} 
for all $t,s\geq 0$ and for all $(\xt{x},\xt{y})\in\mathbb{R}^{2d}$, and where space and time power spectrum is given by
\begin{equation}\label{spectraldens}
\widehat{R}(\omega,\xt{p})=\frac{2\mathfrak{g}(\xt{p})\widehat{R}_0(\xt{p})}{\omega^2+\mathfrak{g}^2(\xt{p})}.
\end{equation}
Moreover, according to \cite[Theorem 1.4.1 pp. 20]{adlertaylor} and the fact that
\[\mathbb{E}\big[\big(V(t_1,x) -V(t_2,y)\big)^2\big]^{1/2} \leq C \left(\int d\xt{p}\widehat{R}_0(\xt{p})\right)\big(\lvert t_1-t_2\rvert+\lvert \xt{x}-\xt{y}\rvert \big),\quad\forall(t_1,t_2,\xt{x},\xt{y})\in[0,T]^2\times K^2,\]
the random potential $V$ defined by \eqref{vtx} is continuous and bounded with probability one on each compact subset $K$ of $\mathbb{R}_+\times\mathbb{R}^d$. Finally, according to the shape of the spatial power spectrum \eqref{rtilde}, we have the following proposition which will be used in the proof of Theorem \ref{thasymptotic} and Theorem \ref{thasymptotic3} based on the perturbed-test-function method.
\begin{prop}\label{propmar} 
Let 
\begin{equation}\label{filtration}\mathcal{F}_t=\sigma(\widehat{V}(s,\cdot),s\leq t)\end{equation}
be the $\sigma$-algebra generated by $(\widehat{V}(s,\cdot), s\leq t)$. We have 
\begin{equation}\label{markovesp}
\mathbb{E}\big[ \widehat{V}(t+h,\cdot) \vert \mathcal{F}_t\big]=e^{-\mathfrak{g}(\xt{p})h}\widehat{V}(t,\cdot)\end{equation}
and for all $(\varphi,\psi)\in \mathcal{H}_m\times \mathcal{H}_m$
\begin{equation}\label{markovvar}\begin{split}
\mathbb{E}\Big[ \big<\widehat{V}(t+h),\varphi\big>_{\mathcal{H}'_m,\mathcal{H}_m} \big<\widehat{V}(t+h),\psi\big>_{\mathcal{H}'_m,\mathcal{H}_m}&-  \mathbb{E}\big[  \big<\widehat{V}(t+h),\varphi\big>_{\mathcal{H}'_m,\mathcal{H}_m} \vert \mathcal{F}_t\big]\mathbb{E}\big[  \big<\widehat{V}(t+h),\psi\big>_{\mathcal{H}'_m,\mathcal{H}_m}\vert \mathcal{F}_t\big]  \Big\vert \mathcal{F}_t\Big]\\
&=\int d\xt{p}\,\, \varphi(\xt{p})\psi(\xt{-p})\widehat{R}_0(\xt{p})\left(1-e^{-2\mathfrak{g}(\xt{p})h} \right).
\end{split}\end{equation}
\end{prop}
The proof of Proposition \ref{propmar} is standard but we still prove it in Section \ref{proofpropmar} just to show how to obtain these results in a weak formulation.

\subsection{Slowly decorrelating assumption} \label{SDCAsec}

In this paper we are interested in the Schrödinger equation with a random potential with long-range correlations. Let us introduce some additional assumptions on the spectral gap $\mathfrak{g}$ of the spatial power spectrum \eqref{rtilde} in order to give slowly decaying correlation properties to the random potential $V$ defined by \eqref{vtx}.

Let us note that for all fixed $t\geq 0$, the random field $V(t,\cdot)$ has spatial slowly decaying correlations. In fact, if we freeze the temporal variable, the autocorrelation function of the random potential $V(t,\cdot)$ is given by
\[R(t,\xt{x})=\E[V(t,\xt{x}+\xt{y})V(t,\xt{y})]=\frac{1}{(2\pi)^d}\int d\xt{p}  \widehat{R}_0(\xt{p})e^{i\xt{p}\cdot\xt{x}}\] 
where $\widehat{R}_0(\xt{p})$ is assumed to have a singularity in $0$, so that $R(t,\cdot)\not\in L^1(\mathbb{R}^d)$. As a result,  $(V(t))_{t\geq 0}$ models a family of random fields on $\mathbb{R}^d$ with spatial long-range correlations which evolves with respect to time. However, since  \eqref{schrodingereq0} is a time evolution problem, we have to take care of the evolution of the random potential $V$ with respect to the temporal variable. In fact, if $V$ has rapidly decaying correlation in time, $(V(t_1),V(t_2))$ has now rapidly decaying spatial correlations, and the evolution problem \eqref{schrodingereq0} behaves like in the mixing case addressed in \cite{bal6}.  As a result, even if at each fixed time the spatial correlations are slowly decaying, the resulting time evolution problem behaves as if the random potential has rapidly decaying correlations. Consequently, we have to introduce a long-range correlation assumption with respect to the temporal variable. 

For the sake of simplicity, throughout this paper we assume that
\begin{equation}\label{SDCAreg}
\mathfrak{g}(\xt{p})=\nu \lvert \xt{p}\rvert ^{2\beta}\quad\text{and}\quad \widehat{R}_0(\xt{p})=\frac{a(\xt{p})}{\lvert\xt{p}\rvert^{d+2(\alpha-1)}},
\end{equation}
where $a$ is a positive continuous function decaying rapidly at infinity and such that $a(0)>0$. Moreover, we assume that $\beta\in(0,1/2]$, $\alpha\in(1/2,1)$, and $\alpha+\beta>1$, so that 
\begin{equation}\label{SDCA}
\int d\xt{p}\frac{\widehat{R}_0(\xt{p})}{\mathfrak{g}(\xt{p})}=+\infty,\end{equation}
since
\begin{equation}\label{theta}\frac{\widehat{R}_0(\xt{p})}{\mathfrak{g}(\xt{p})}\underset{\xt{p}\to 0}{\sim} \frac{a(0)}{\lvert \xt{p}\rvert^{d+\theta}}\quad \text{with}\quad\theta=2(\alpha+\beta-1)\in(0,1).\end{equation}
A similar configuration has already been considered in \cite{bal2} to study the propagation of the wave field $\phi_\e$ defined by \eqref{phis} in a random media with long-range correlations. As a result, these assumptions permit to model a random field $V(t,\xt{x})$ with spatial long-range correlations for each fixed time $t\geq 0$, and with slowly decaying correlations in time because of the following relation :
\begin{equation}\label{SDCAcor}\forall (s,\xt{x},\xt{y})\in \mathbb{R}_+\times\mathbb{R}^{2d}\quad \int_0^{+\infty} dt \big\lvert\E\big[V(t+s,\xt{x}+\xt{y})V(s,\xt{y})\big]\big\rvert=+\infty \Longleftrightarrow \int d\xt{p}  \frac{\widehat{R}_0(\xt{p})}{\mathfrak{g}(\xt{p})} =+\infty.  \end{equation}
If fact, if \eqref{SDCA} holds, we have
\[\lim_{A\to+\infty}\int_0^{A} dt \big\lvert\E\big[V(t+s,\xt{x}+\xt{y})V(s,\xt{y})\big]\big\rvert\geq  \lim_{A\to+\infty}\int d\xt{p}  \frac{\widehat{R}_0(\xt{p})}{\mathfrak{g}(\xt{p})} (1-e^{-\mathfrak{g}(\xt{p})A})-\int d\xt{p}  \frac{\widehat{R}_0(\xt{p})}{\mathfrak{g}(\xt{p})}\lvert e^{i\xt{p}\cdot\xt{x}}-1\rvert,
\]
and the converse implication is obvious by taking $\xt{x}=0$. Consequently, throughout this paper we say that the family $(V(t))_{t\geq 0}$ of random fields with spatial long-range correlations has slowly decaying correlations in time if \eqref{SDCA} holds, and rapidly decaying correlations in time otherwise.

\subsection{Wigner transform}\label{wignersection}

In this paper we study wave decoherence phenomena, using the following Wigner transform \eqref{wignertrans} of the wave field \eqref{phis} satisfying the Schrödinger equation \eqref{schrodingereq}, happening on different propagation scales $s$. In this paper we consider the Wigner transform of the field $\phi_\e$, averaged with respect to the randomness of the initial data, defined by:
\begin{equation}\label{wignertrans}\begin{split}
W_\e (t,\xt{x},\xt{k})&=\frac{1}{(2\pi)^d}\int_{\mathbb{R}^d\times S}d\xt{y}\mu(d\zeta)e^{i\xt{k}\cdot\xt{y}} \phi_\e \Big(t,\xt{x}-\e^{s-s_c}\frac{\xt{y}}{2},\zeta\Big)\overline{\phi_\e \Big(t,\xt{x}+\e^{s-s_c}\frac{\xt{y}}{2},\zeta\Big)}\\
&=\frac{1}{(2\pi)^d}\int_{\mathbb{R}^d\times S}d\xt{y}\mu(d\zeta)e^{i\xt{k}\cdot\xt{y}} \phi \Big(\frac{t}{\e^s},\frac{\xt{x}}{\e^s}-\frac{\xt{y}}{2\e^{s_c}},\zeta\Big)\overline{\phi \Big(\frac{t}{\e^s},\frac{\xt{x}}{\e^s}+\frac{\xt{y}}{2\e^{s_c}},\zeta\Big)},\\
\end{split}\end{equation}
where $\boldsymbol{s_c\in[0,s]}$ \textbf{is the spatial correlation parameter}, and $(S,\mathcal{S},\mu(d\zeta))$ is a probability space. Let us remark that the scaled Wigner transform  \eqref{wignertrans} is a real-valued function. We discuss below the reason of introducing this probability space, and we refer to \cite{gerard, lions} for the basic properties of the Wigner distribution.  The scaled Wigner transform \eqref{wignertrans}, which is somehow the Fourier transform of the correlation function in space of the wave field $\phi$ (resp. $\phi_\e$) around $\xt{x}$, measures the degree of correlation of  the wave field. In other words, it captures the evolution of the degree of correlation of the wave field $\phi$ over a spatial correlation scale of order $\e^{-s_c}$, on the microscopic propagation scale. Or in the same way, it captures the evolution of the degree of correlation of the rescaled wave field $\phi_\e$ over a spatial correlation scale of order $\e^{s-s_c}$, on the macroscopic propagation scale $\e^{-s}$. 

The \emph{loss of coherence} of the wave field $\phi$ satisfying \eqref{schrodingereq0} (or equivalently the \emph{wave decoherence}) corresponds to the evolution in time of the momentum of the Wigner transform $W_\e$. In other words, if the momentum of the Wigner transform is preserved during the propagation there is no wave decoherence, and in the opposite case the wave decoherence mechanism is described by the evolution of the momentum. 

In this paper we consider a rescaled version of the Wigner transform involving the two parameters $s$ and $s_c$. The propagation scale parameter $s$ has been introduced in Section \ref{intro} and characterizes the order of magnitude of the propagation times and the propagation distances on which we consider the wave field $\phi$ satisfying \eqref{schrodingereq0}. The spatial correlation parameter $s_c$ characterizes the spatial scale on which we study the evolution of the degree of correlation of the wave field $\phi$. Let us note that the cases $s_c<s$ study the local loss of coherence of the wave field, while the case $s=s_c$ study the nonlocal loss of coherence of the wave field. We will see in Section \ref{sectionint} and Section \ref{sections1} that for a given propagation scale parameter $s$ the loss of coherence of the wave field $\phi$ can be observed at a particular spatial correlation parameter $s_c$ depending on $s$ (see \eqref{sc}). The larger the propagation scale parameter is the shorter the decoherence scale parameter is. In other words, depending on the propagation time and propagation distance we adjust the spacial correlation scale characterized by $s_c$ to exhibit the loss of coherence (see Figure \ref{schema}).  For instance, according to \cite{gomez} no significant loss of coherence  of the wave can be exhibited on the correlation scale $s_c=0$ before $s=1$. Let us note that for a rapidly deccorelating potential $V$, no wave decoherence effects can be observed except for the radiative transfer scaling $s_c=0$ and $s=1$ \cite{bal2}. The reason will be explain formally in Section \ref{sectionint}.  

However, to observe decoherence effects of the field $\phi_\e$ on the spatial correlation scale of order $\e^{s-s_c}$, we need a proper initial condition $\phi_{0,\e}$ in  \eqref{schrodingereq}, which oscillates at the same scale (see Figure \ref{wavefig}). Moreover, a natural way to introduce randomness on the initial condition is as follow. Let $S=\mathbb{R}^d$ and $\mu(\zeta)$ be a nonnegative rapidly decreasing function such that $\|\mu\|_{L^1(\mathbb{R}^d)}=1$, and so that $(\mathbb{R}^d,\mathcal{B}(\mathbb{R}^d),\mu(d\zeta))$ is a probability space. Throughout this paper we assume that the initial condition $\phi_{0,\e}$ in \eqref{schrodingereq} is given by  
\begin{equation}\label{initcond}
\phi_{0,\e}(\xt{x})=\phi_0(\xt{x})\exp(i\zeta\cdot \xt{x}/\e^{s-s_c}).
\end{equation}  
This initial condition represents a plane wave with initial propagation direction $\zeta\in\mathbb{R}^d$, oscillating on the spatial scale $\e^{s-s_c}$, and with amplitude or envelope $\phi_0$. The initial direction $\zeta$ of the wave is distributed according to $\mu(d\zeta)$, so that the Wigner transform \eqref{wignertrans} is average according to the distribution of the initial direction of the wave. Let us note that the spatial frequency of the initial condition ($\sim\e^{-(s-s_c)}$) is low compared to the one of the random medium ($\sim \e^{-s}$) on the macroscopic scale $\e^{-s}$. In rapidly deccorelating random media such low spatial frequency sources do not interact with the random medium, but as we will in Section \ref{sectionint}, this kind of initial conditions interact strongly with slowly decorrelating random media. This result can be useful in passive imaging of a target in slowly deccorelating random media \cite{garnier2}.

\begin{figure}\begin{tabular}{cc}
\includegraphics[scale=0.25]{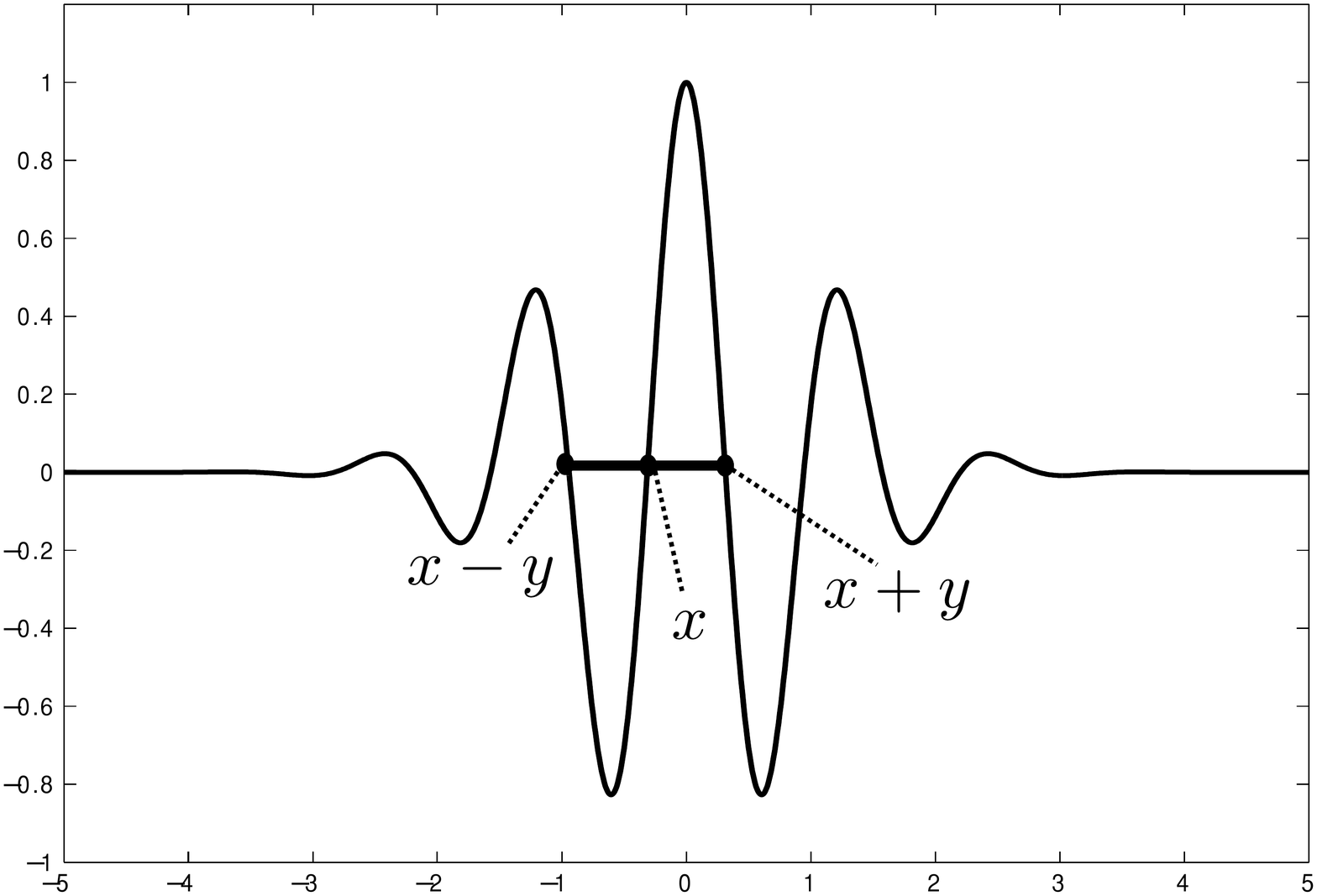}&\includegraphics[scale=0.25]{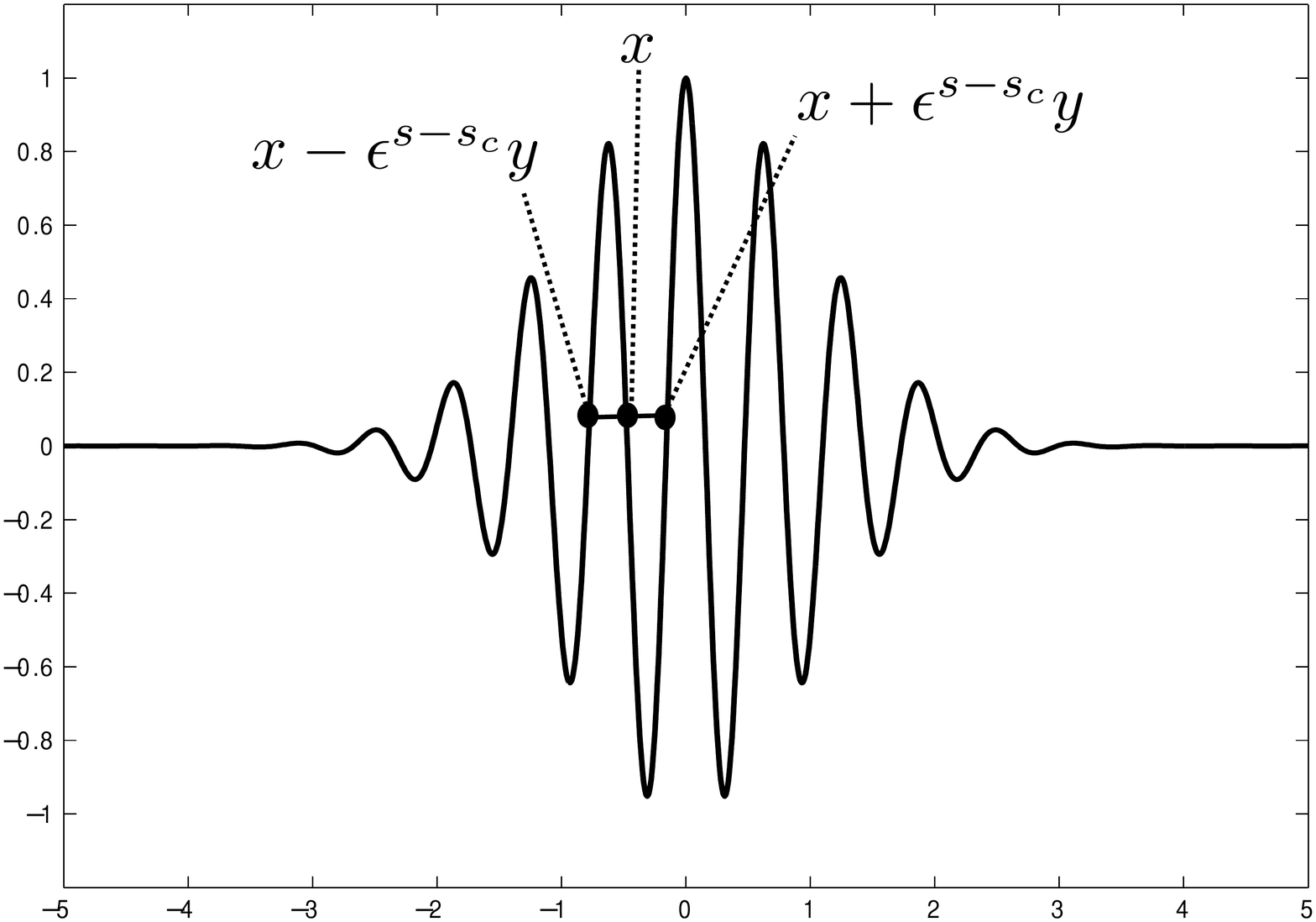}\\
$(a)$&$(b)$\\
\includegraphics[scale=0.25]{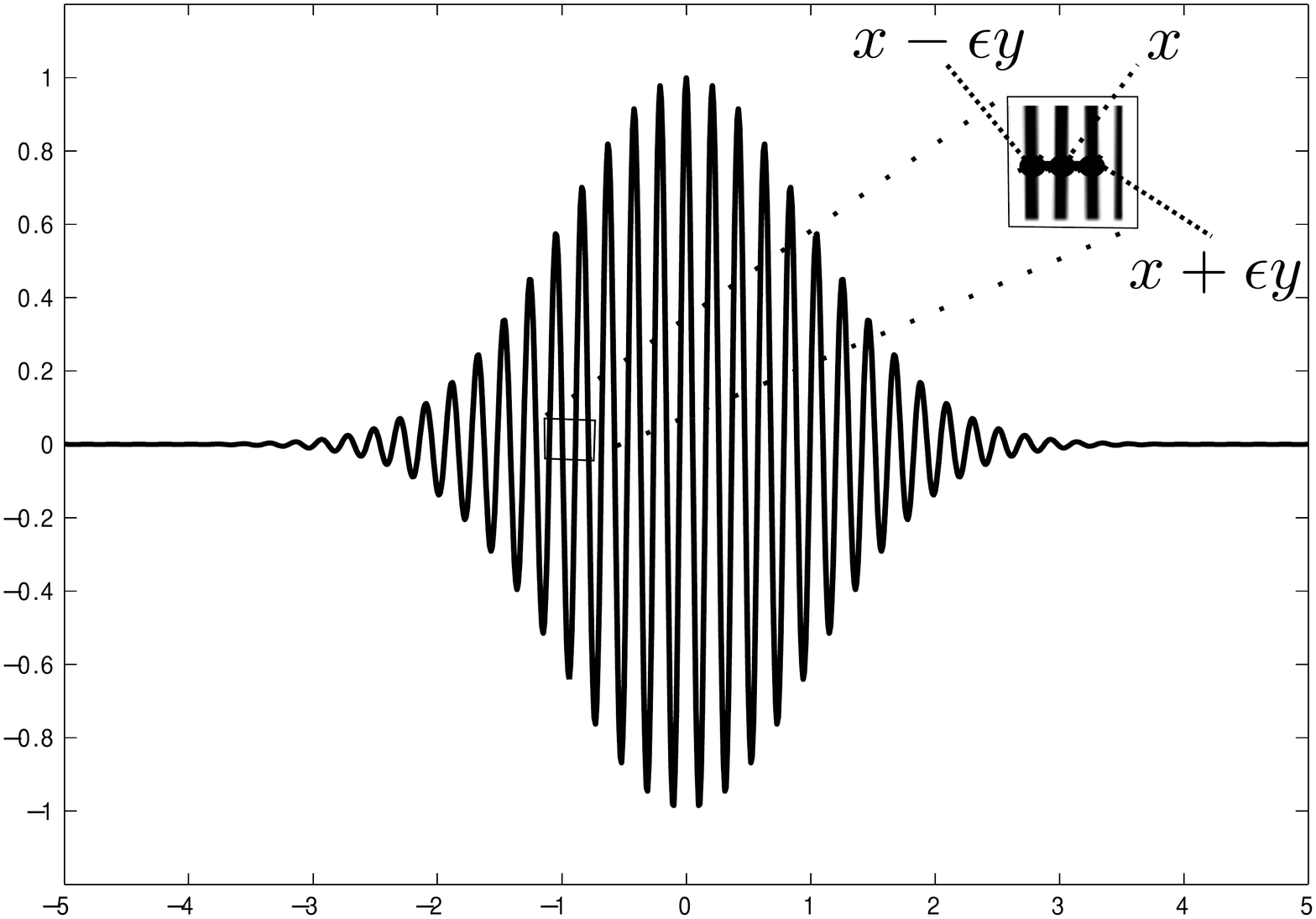}& \\
$(c)$ & 
\end{tabular}
\caption{\label{wavefig} Illustration of the initial condition \eqref{initcond}. $(a)$ and $(b)$ represent low spatial frequency initial conditions compared to the spatial frequency of the random medium $\e^{-s}$ on the macroscopic scale  $\e^{-s}$. $(a)$ represents the case $s_c=s$ and has a spatial frequency of order $1$ on the macroscopic scale $\e^{-s}$ (and a spatial frequency of order $\e^s$ on the microscopic scale). $(b)$ represents the case $s_c<s$ and has a spatial frequency of order $\e^{-(s-s_c)}$ on the macroscopic scale  $\e^{-s}$ (and a spatial frequency of order $\e^{s_c}$ on the microscopic scale). $(c)$ represents the case $s_c=0$ and $s=1$, and has a spatial frequency of order $\e^{-1}$ on the macroscopic scale $\e^{-1}$ (and a spatial frequency of order $1$ on the microscopic scale).}
\end{figure} 

The main reason to introduce this additional randomness through the initial data $\phi_{0,\e}$ is to make possible the weak convergence of the initial Wigner transform $W_{\e}(0)$ in $L^2(\mathbb{R}^{2d})$, which stands for the set of real-valued square-integrable functions equipped with the following inner product
\[\big<\lambda,\mu\big>_{L^2(\mathbb{R}^d)}=\int_{\mathbb{R}^{2d}}d\xt{x}d\xt{k}\,f(\xt{x},\xt{k})g(\xt{x},\xt{k}),\quad \forall (\lambda,\mu)\in L^2(\mathbb{R}^{2d})\times L^2(\mathbb{R}^{2d}).\] 
Consequently, we have  
\[
\forall \lambda\in L^2(\mathbb{R}^{2d}),\quad \lim_\e\big<W_{\e}(0),\lambda\big>_{L^2(\mathbb{R}^{2d})}= \big<W_{0},\lambda\big>_{L^2(\mathbb{R}^{2d})},
\]
where
\begin{equation}\label{wignerinitcond}W_{\e}(0,\xt{x},\xt{k})=W_{0,\e}(\xt{x},\xt{k})=\frac{1}{(2\pi)^d}   \int_{\mathbb{R}^{2d}}d\xt{y}\mu(d\zeta)e^{i(\xt{k}-\zeta)\cdot\xt{y}} \phi_0(\xt{x}-\e^{s-s_c}\xt{y}/2)\overline{\phi_0(\xt{x}+\e^{s-s_c}\xt{y}/2)},\end{equation}
with
\begin{equation}\label{w01} W_{0}(\xt{x},\xt{k})=\lvert \phi_0(\xt{x})\rvert^2 \mu(\xt{k}), \quad \text{if } s_c<s,\end{equation}
 and 
\begin{equation}\label{w02}W_{0}(\xt{x},\xt{k})=\frac{1}{(2\pi)^d}   \int_{\mathbb{R}^{2d}}d\xt{y}\mu(d\zeta)e^{i(\xt{k}-\zeta)\cdot\xt{y}} \phi_0(\xt{x}-\xt{y}/2)\overline{\phi_0(\xt{x}+\xt{y}/2)},\quad \text{if }s=s_c.\end{equation}
Consequently, thanks to the Banach-Steinhaus Theorem, $W_{0,\e}$ is uniformly bounded in $L^2(\mathbb{R}^{2d})$ with respect to $\e$. We need such a convergence on the initial Wigner transform $W_{\e}(0)$ since we study the Wigner transform \eqref{wignertrans} in $W_\e$ in $L^2(\mathbb{R}^{2d})$ equipped with the weak topology. As it will be discussed in Section \ref{sectionint}, it is not possible to expect a convergence result in $L^2(\mathbb{R}^{2d})$ equipped with the strong topology (except for the case $s=s_c$). 

The Wigner distribution \eqref{wignertrans} satisfies the following evolution equation
\begin{equation}\label{wignereq}\begin{split}
\partial_t W_\e(t,\xt{x},\xt{k})&+\e^{s_c}\xt{k}\cdot\nabla_{\xt{x}}W_\e(t,\xt{x},\xt{k})=\\
&\e^{(1-\ga)/2-s}\int_{\mathbb{R}^d}\frac{\widehat{V}\Big(\frac{t}{\e^{s+\gamma}},d\xt{p}\Big)}{(2\pi)^d i} e^{i\xt{p}\cdot\xt{x}/\e^s}\Big( W_\e\Big(t,\xt{x},\xt{k}-\frac{\xt{p}}{2\e^{s_c}}\Big)-W_\e\Big(t,\xt{x},\xt{k}+\frac{\xt{p}}{2\e^{s_c}}\Big)\Big),
\end{split}\end{equation}
with initial conditions $W_\e(0,\xt{x},\xt{k})=W_{0,\e}(\xt{x},\xt{k})$, and where $W_{0,\e}$ is defined by \eqref{wignerinitcond}. The previous equation \eqref{wignereq} can be recast in the weak sense as follows : 
\[
\big<W_\e(t),\lambda\big>_{L^2(\mathbb{R}^{2d})}-\big<W_\e(0),\lambda\big>_{L^2(\mathbb{R}^{2d})}=\int_0^t \big<W_\e(u),\e^{s_c} \xt{k}\cdot\nabla_{\xt{x}}\lambda+\e^{(1-\ga)/2-s}\mathcal{L}_\e\lambda(u)\big>_{L^2(\mathbb{R}^{2d})}du,
\]
 for all $\lambda\in \mathcal{S}(\mathbb{R}^{2d})$, where $\mathcal{S}(\mathbb{R}^{2d})$ is the Schwartz space and stands for the space of rapidly decaying functions. Here
\begin{equation}\label{lepsilon}
\mathcal{L}_\e\lambda(t,\xt{x},\xt{k})=\frac{1}{(2\pi)^d i}\int_{\mathbb{R}^d}\widehat{V}\Big(\frac{t}{\e^{s+\ga}},d\xt{p}\Big) e^{i\xt{p}\cdot\xt{x}/\e^s}\Big( \lambda\Big(\xt{x},\xt{k}-\frac{\xt{p}}{2\e^{s_c}}\Big)-\lambda\Big(\xt{x},\xt{k}+\frac{\xt{p}}{2\e^{s_c}}\Big)\Big).
\end{equation}

Let us assume that $V=0$, so that the Wigner transform is given by $W_\e(t,\xt{x},\xt{k})=W_0(\xt{x}-\e^{s_c}t\xt{k},\xt{k})$. Therefore, the momentum of $W_\e$ is preserved during the propagation, meaning there is no variation of the momentum with respect to time, that is there is no wave decoherence.  The dispersion term $\xt{k}\cdot \nabla_\xt{x}$ is of order $\e^{s_c}$ so that $W_\e$ captures the wave dispersion only for $s_c=0$. The transfer equation \eqref{wignereq}  describes the loss of coherence of the field $\phi_\e$ through the random operator $\mathcal{L}_\e W_\e(t)(t,\xt{x},\xt{k})$. However, depending on the spatial scale of observation the loss of coherence of the wave may not be significant. In fact, according to \cite{gomez} no significant wave decoherence can be exhibited on the correlation scale $s_c=0$ before $s=1$.

Let us introduce some notations which are used in Section \ref{sectionint} and Section \ref{sections1}. Let
\[\mathcal{B}_{r}=\left\{\lambda \in L^2(\mathbb{R}^{2d}), \|\lambda\|_{L^2(\mathbb{R}^{2d})}\leq r\right\},\quad \text{with}\quad r=\sup_{\e}\|W_{0,\e}\|_{L^2(\mathbb{R}^{2d})}<+\infty, \] 
be the closed ball with radius $r$, and $\{g_n, n\geq 1\}$ be a dense subset of $\mathcal{B}_{r}$. We equip  $\mathcal{B}_{r}$ with the distance $d_{\mathcal{B}_{r}}$ defined by
\begin{equation}\label{defmetric}d_{\mathcal{B}_{r}}(\lambda, \mu)=\sum_{j=1}^{+\infty}\frac{1}{2^j}\left\lvert\big<\lambda-\mu,g_n\big>_{L^2(\mathbb{R}^{2d})}\right\rvert,\quad \forall (\lambda,\mu)\in\mathcal{B}_{r}\times\mathcal{B}_{r},\end{equation}
so that $(\mathcal{B}_{r} ,d_{\mathcal{B}_{r}})$ is a compact metric space. Therefore, $(W_\e)_\e$ is a family of process with values in $(\mathcal{B}_{r} ,d_{\mathcal{B}_{r}})$, since $\|W_\e (t)\|_{L^2(\mathbb{R}^{2d})}=\|W_{0,\e}\|_{L^2(\mathbb{R}^{2d})}$. The topology generated by the metric $d_{\mathcal{B}_r}$  is equivalent to the weak topology on $L^2(\mathbb{R}^{2d})$ restricted to $\mathcal{B}_r$.

The three following sections describe in a chronological order the effects produced by the random medium on the wave propagation.

\section{Phase Modulation Scaling $s=1/(2\kappa_\ga)$}\label{sectionsk}

This section describes the first effects caused by the small random fluctuations of the medium on the wave propagation. The following theorem presents the asymptotic behavior of the phase of $\phi_\e$ solution of \eqref{schrodingereq}. Theorem \ref{thphase} has been shown \cite{bal2} in the case $\ga=0$ and $s_c=0$, but nevertheless, its proof remains the same as the one of \cite[Theorem 1.2]{bal2}. We state this result in order to provide a complete and self-contained presentation of the wave propagation in long-range random media.

\begin{thm}\label{thphase}
Let us note 
\[\kappa_0=\frac{\alpha+2\beta-1}{2\beta}\in(1/2,1),\] 
 and 
 \[ \kappa_\ga=\frac{\kappa_0}{1-\ga\Big(\frac{\alpha+\beta-1}{\beta}\Big)}\quad \text{for }\ga\in[0,1),\]
and let us consider the process $\widehat{\zeta}_{\kappa_\ga,\e}(t,\xt{k})$ defined by 
\begin{equation}\label{zeta}  \widehat{\zeta}_{\kappa_\ga,\e}(t,\xt{k})=\frac{1}{\e^{d(s-s_c)}}\widehat{\phi}_{\e}\Big(t,\frac{\xt{k}}{\e^{s-s_c}}\Big)e^{i \lvert \xt{k}\rvert^2 t/(2\e^{s-2s_c})},\quad \text{with}\quad s=1/(2\kappa_\ga),\text{ and }s_c\leq s,\end{equation}
where $\phi_\e$ satisfies \eqref{schrodingereq} with initial data \eqref{initcond}. Under the long-range correlation assumption in time \eqref{SDCA}, where $\alpha+\beta>1$, the process $\widehat{\zeta}_{\kappa_\ga,\e}(t,\xt{k})$  converges in distribution to
\[\widehat{\zeta}(t,\xt{k})=\widehat{\zeta}_0(\xt{k})\exp\Big(i\sqrt{D(\alpha,\beta,\xt{k})}B_{\kappa_0}(t)\Big),\]
 for each $t\geq0$ and $\xt{k}\in\mathbb{R}^d$, where $(B_{\kappa_0}(t))_{t\geq 0}$ is a standard fractional Brownian motion with Hurst index $\kappa_0$, and
\[\widehat{\zeta}_0(\xt{k})=\widehat{\phi}_0(\zeta-\xt{k})\quad\text{if }s_c=s,\quad\text{and}\quad\widehat{\zeta}_0(\xt{k})=\phi_0(0)\delta(\zeta-\xt{k})\quad\text{otherwise}.\]
Moreover,
\[D(\alpha,\beta,\xt{k})=\frac{a(0)}{(2\pi)^d\kappa_0(2\kappa_0-1)}\int_0^{+\infty}d\rho \frac{e^{-\nu\rho}}{\rho^{2\alpha-1}}\int_{\mathbb{S}^{d-1}}dS(u) e^{i\lvert\xt{k}\rvert\rho u\cdot \xt{e}_1}\quad\text{if }\beta=\frac{1}{2},\,\ga=0,\text{ and }s_c=0,\]
and
\[D(\alpha,\beta,\xt{k})=D(\alpha,\beta)=\frac{a(0)\Omega_d}{(2\pi)^d\kappa_0(2\kappa_0-1)}\int_0^{+\infty}d\rho \frac{e^{-\nu\rho^{2\beta}}}{\rho^{2\alpha-1}}\quad \text{otherwise},\]
where, $\Omega_d$ is the surface area of the unit sphere in $\mathbb{R}^d$, and $\xt{e}_1\in \mathbb{S}^{d-1}$.
\end{thm}

In Theorem \ref{thphase}, $\widehat{\zeta}_0$ represents the initial direction of the wave, and $\exp(i\sqrt{D(\alpha,\beta,\xt{k})}B_{\kappa_0}(t))$ represents the random phase modulation induces by the slowly decorrelating perturbations of the medium. The scaling $\xt{k}/\e^{s-s_c}$ in \eqref{zeta} corresponds to the order of the spacial frequency of the initial condition \eqref{initcond}. The order of the spacial frequency of the initial condition does not play a significant role in Theorem \ref{thphase} but it will in Section \ref{sectionint} and Section \ref{sections1}. In fact in these sections the order of the spacial frequency of the initial condition \eqref{initcond} plays a significant role and is linked to the correlation scale parameter $s_c$ of the Wigner transform \eqref{wignertrans} to exhibit the loss of coherence. The order of the spacial frequency is taken into account in Theorem \ref{thphase} in order to give a complete picture of the wave propagation phenomena (see Figure \ref{schema}).

As a result, in long-range random media macroscopic effects may happen on the field $\phi_\e$ at a shorter scale $s=1/(2\kappa_\ga)<1$ without induced loss of coherence of the field $\phi_\e$. These effects are just a phase modulation given by a fractional Brownian motion with Hurst index $\kappa_0$, which depends on the statistic properties of the random potential $V$. However, let us note that the propagation scale parameter $s=1/(2\kappa_\ga)<1$ is "universal" in the sense that a random phase modulation appears on the wave whatever the order of the low-frequency initial condition characterized by $s_c$, that is for all $s_c\in [0, s=1/(2\kappa_\ga)]$.

\cite{bal2} is the first paper showing a qualitative difference between the random effects induced on a wave propagating in long range and in rapidly decorrelating random media in time $\eqref{SDCAcor}$, for propagation media of dimension strictly greater than $1$. In fact, it has been shown in \cite{bal2} that the field $\phi_\e$ propagating in a rapidly decorrelating medium does not evolve before the scale $s=1$, more precisely the phase and the phase space energy evolve at the same propagation scale $\e^{-1}$ ($s=1$), so that no significant wave decoherence can be observed before this propagation scale. In rapidly decorrelating random media the scale $s=1$ is "universal", in the sense that it does not depend on the statistic of the random potential $V$. 

Theorem \ref{thphase} shows that the phase of the wave field $\phi_\e$ exhibits non trivial stochastic phenomena for $s=1/(2\kappa_\ga)$. As a result, we should expect for larger propagation scale parameters $s>1/(2\kappa_\ga)$  that the random oscillations  of the wave field $\phi_\e$ will evolve faster and faster up to break the wave coherence. We show in Section \ref{sectionint} and Section \ref{sections1} that the loss of coherence appears first on the large spatial correlation scales, and then as $s$ increase, it is transmitted to smaller spatial correlation scales (see Figure \ref{schema}).

\section{Wave Decoherence for $s\in(1/(2\kappa_\ga),1)$}\label{sectionint}

This section presents the main result of this paper, it describes the loss of coherence of the wave field $\phi_\e$ solution of \eqref{schrodingereq} occurring after the onset of the random phase modulation described in Section \ref{sectionsk}. On propagation scales $\e^{-s}$ with $s>1/(2\kappa_\ga)$ but strictly less than $1$, the random phase modulation described in the asymptotic $\e\to 0$ in the Section \ref{sectionsk} begins to oscillate very fast up to break the wave coherence, and produce momentum diffusion effect. However, according to  \cite{gomez}  the wave decoherence does not take place for $s_c=0$. To study this diffusion phenomenon we need to consider the scaled Wigner transform \eqref{wignertrans} to capture the wave decoherence for spatial correlation parameters $s_c>0$. Using the notation introduced in Section \ref{wignersection}, we have the following results. 

In Theorem \ref{thasymptotic}, Theorem \ref{thasymptotic3}, and Theorem \ref{thasymptotic2}, we show that the good spatial correlation parameter to observe the loss of coherence is 
\[s_c=(1-s)/\theta,\] 
so that for $s_c\leq s$, we have $s\geq 1/(1+\theta)\geq1/(2\kappa_\ga)$, where $\theta$ is defined by \eqref{theta} and $\kappa_\ga$ in Theorem \ref{thphase}. As a result, no decoherence effect can be observed before the propagation scale $s=1/(1+\theta)$ for any spatial correlation parameter $s_c\leq s$.  Theorem \ref{thasymptotic} below deals with the case $s>1/(1+\theta)$ for which one can describe the wave decoherence in term of a fractional diffusion, while Theorem \ref{thasymptotic3} deals with the critical case $s_c=s=1/(1+\theta)$.

\begin{thm}\label{thasymptotic}
Let us assume that \eqref{SDCA} holds. For $s\in(1/(2\kappa_\ga),1)$, and 
\begin{equation}\label{sc}s_c=\frac{1-s}{\theta}<s,\end{equation}
where $\theta\in(0,1)$ is defined by \eqref{theta}, the family of scaled Wigner transform $(W_\e)_{\e\in(0,1)}$ defined by \eqref{wignertrans} and solution of the transport equation \eqref{wignereq}, converges in probability on 
$\mathcal{C}([0,+\infty), (\mathcal{B}_{r} ,d_{\mathcal{B}_{r}}))$ as $\e \to 0$ to a limit denoted by $W$. More precisely, for all $T>0$ and for all $\eta>0$,
\[\lim_{\e\to 0}\Pro\left(\sup_{t\in[0,T]}d_{\mathcal{B}_r}(W_\e(t),W(t))>\eta\right)=0,\]
where $W$ is the unique solution uniformly bounded in $L^2(\mathbb{R}^{2d})$ of the fractional diffusion equation
\begin{equation}\label{radtranseq}
\partial_t W=-\sigma(\theta)(-\Delta_\xt{k})^{\theta/2}W,
\end{equation} 
with $W(0,\xt{x},\xt{k})=W_0(\xt{x},\xt{k})\in L^2(\mathbb{R}^{2d})$ defined by \eqref{w01}. 
Here, $(-\Delta_\xt{k})^{\theta/2}$ is the fractional Laplacian with Hurst index $\theta\in(0,1)$, and
\[\sigma(\theta)=\frac{2a(0) \theta\Gamma(1-\theta)}{(2\pi)^d}\int_{\mathbb{S}^{d-1}}dS(\xt{u})\lvert\xt{e}_1\cdot \xt{u} \rvert^\theta\]
with $\xt{e}_1\in \mathbb{S}^{d-1}$ and $\Gamma(z)=\int_0^{+\infty} t^{1-z}e^{-t}dt$.
Moreover, $W$ is given by the following formula
\begin{equation}\label{formula}W(t,\xt{x},\xt{k})=\frac{1}{(2\pi)^d}\int d\xt{q} \exp(i\xt{k}\cdot\xt{q}-\sigma(\theta) \lvert \xt{q} \rvert^\theta t)\widehat{W}^{\xt{k}}_0(\xt{x},\xt{q}),\end{equation}
where $\widehat{W}^{\xt{k}}_0$ stands for the Fourier transform of $W_0$ with respect to the variable $\xt{k}$.
\end{thm}

Let us note that we cannot expect a convergence on $L^2(\mathbb{R}^{2d})$ equipped with the strong topology. In fact, the following conservation relation $\|W_\e(t)\|_{L^2(\mathbb{R}^{2d})}=\|W_{0,\e}\|_{L^2(\mathbb{R}^{2d})}$ is not true anymore for the limit $W$. Moreover, let us note that the Wigner distribution $W_\e$ is self-averaging as $\e$ goes to $0$, that is the limit $W$ is not random anymore. This self-averaging phenomenon of the Wigner distribution has already been observed in several studies \cite{bal3, bal,bal4,gomez} for $s=1$, and is very useful for applications. The proof of Theorem \ref{thasymptotic} is given in Section \ref{proof} and is based on an asymptotic analysis using perturbed-test-function and martingale techniques.

Equation \eqref{radtranseq} describes the loss of coherence of the field $\phi_\e$ for the particular spatial correlation parameter $s_c$ defined by \eqref{radtranseq} through a momentum diffusion. This fractional diffusion exhibits a damping term obeying to a power law with exponent $\theta\in(0,1)$ describing the decoherence rate of the field wave $\phi_\e$. An important point is that the wave decoherence mechanism is deterministic, it does not depend on the particular realization of the random medium. Finally, let us note that $W$ does not evolves in $\xt{x}$. In fact, in \eqref{wignereq} the dispersion term $\xt{k}\cdot \nabla_\xt{x}$ is of order $\e^{s_c}$. As a result, for $s_c>0$ the dispersion is not captured by our scaled Wigner transform \eqref{wignertrans} and is small compared to the momentum diffusion mechanism. The dispersion effect can be captured by the scaled Wigner transform only when $s_c=0$ (see Section \ref{sections1}).

The following theorem investigates the special case $s_c=s=1/(1+\theta)$, with either $\ga>0$ or $\beta<1/2$. This special case studies the decoherence of the wave envelop $\phi_0$ itself and not the local loss of coherence of the initial condition (see \eqref{initcond}). The case $\ga=0$ and $\beta=1/2$, and $s=s_c$ has been addressed in Theorem \ref{thphase}, since in this particular case $1/(1+\theta)=1/(2\kappa_0)$. Therefore, in this case there is no wave decoherence.

\begin{thm}\label{thasymptotic3}

Let us assume that \eqref{SDCA} holds. For either $\ga>0$ or $\beta<1/2$, and
\[s_c=s=\frac{1}{1+\theta},\]
where $\theta\in(0,1)$ is defined by \eqref{theta}, the family of scaled Wigner transform $(W_\e)_{\e\in(0,1)}$ defined by \eqref{wignertrans} and solution of the transport equation \eqref{wignereq}, converges in distribution on 
$\mathcal{C}([0,+\infty), L^2(\mathbb{R}^{2d}))$ as $\e \to 0$ to a limit $W$ defined by
\[W(t,\xt{x},\xt{k})=\frac{1}{(2\pi)^d}\int d\xt{q} \widehat{W}^\xt{k}_0(\xt{x},\xt{q})\exp\Big(i\xt{k}\cdot\xt{q}+i \int \mathcal{B}_t(d\xt{p})e^{i\xt{p}\cdot\xt{x}}( e^{-i\xt{q}\cdot \xt{p}/2} - e^{i\xt{q}\cdot \xt{p}/2} ) \Big).\]
$W$ is the unique weak solution of the stochastic differential equation
\begin{equation}\label{radtranseq2}\begin{split}
d W(t,\xt{x},\xt{k})=&\,-\sigma(\theta)(-\Delta_\xt{k})^{\theta/2}W(t,\xt{x},\xt{k})\\
&\,+\frac{2ia(0)}{(2\pi)^{d}}\int \mathcal{B}_t(d\xt{p})e^{i\xt{x}\cdot\xt{p}}\Big(W\big(t,\xt{x},\xt{k}-\frac{\xt{p}}{2})-W\big(t,\xt{x},\xt{k}+\frac{\xt{p}}{2})\Big),
\end{split}\end{equation} 
with $W(0,\xt{x},\xt{k})=W_0(\xt{x},\xt{k})\in L^2(\mathbb{R}^{2d})$ defined by \eqref{w02}. 
Here, $(\mathcal{B}_t)_t$ is a real Brownian motion on $\mathcal{H}'_\theta$ the dual space of
\[\mathcal{H}_\theta=\Big\{ \varphi \text{ such that}\quad\varphi(\xt{p})=\overline{\varphi(-\xt{p})}\quad\text{and}\quad  \int\frac{d\xt{p}}{\lvert \xt{p}\rvert ^{d+\theta}}\lvert \varphi(\xt{p})\rvert^2 <+\infty   \Big\},\]
with covariance function
\[ \E\big[ \mathcal{B}_t(\varphi)\mathcal{B}_s(\psi)  \big] =s\wedge t \int\frac{d\xt{p}}{\lvert \xt{p} \rvert^{d+\theta}} \varphi(\xt{p})\overline{\psi(\xt{p})},\quad \forall(\varphi,\psi)\in\mathcal{H}_\theta\times\mathcal{H}_\theta.\]
Moreover, $(-\Delta_\xt{k})^{\theta/2}$ is the fractional Laplacian with Hurst index $\theta\in(0,1)$, and
\[\sigma(\theta)=\frac{2a(0) \theta\Gamma(1-\theta)}{(2\pi)^d}\int_{\mathbb{S}^{d-1}}dS(\xt{u})\lvert\xt{e}_1\cdot \xt{u} \rvert^\theta\]
with $\xt{e}_1\in \mathbb{S}^{d-1}$ and $\Gamma(z)=\int_0^{+\infty} t^{1-z}e^{-t}dt$.
\end{thm}
 
Let us remark that the proof of the weak uniqueness of \eqref{radtranseq2} follows the idea developed in \cite{fannjianguni,weinryb}, and is the same as the one given in Section \ref{proof2} to prove that all the converging subsequences of $(W_\e)_{\e\in(0,1)}$ have the same distribution.

This limiting Wigner transform is random because the wave does not propagate enough to observe a self-averaging of the stochastic phenomena. In fact, as shown in Theorem \ref{thasymptotic} for all $s>1/(1+\theta)$ the limiting Wigner transform is self-averaging and is equal to the expectation of the limiting Wigner transform obtain in the case $s_c=s=1/(1+\theta)$.  

Let us note that the convergence holds on $L^2(\mathbb{R}^{2d})$ equipped with the strong topology. In fact, the conservation relation $\|W_\e(t)\|_{L^2(\mathbb{R}^{2d})}=\|W_{0,\e}\|_{L^2(\mathbb{R}^{2d})}$ is preserved for the limiting process $W$. In this scaling the limit $W$ is a stochastic process, which is given in the Fourier domain by a random phase modulation. As illustrated in Figure \ref{schema}, the random phase modulation of the Wigner transform is caused by the fast phase modulation of the wave field $\phi_\e$ itself and described in Theorem \ref{thphase} in the asymptotic $\e \to0$. For $s\in(1/(2\kappa_\ga),1/(1+\theta))$ the random phase modulation  the wave field $\phi_\e$ oscillates very fast up to produce a phase modulation  on the Wigner transform at the propagation scale $\e^{-s}$ with $s=1/(1+\theta)$. It is the first propagation scale parameter $s$ for which the coherence of the wave is affected. Afterwards, for the propagation scale parameter $s>1/(1+\theta)$ this random phase modulation of the Wigner transform oscillates very fast and then average out to obtain the determinist wave decoherence mechanism described in Theorem \ref{thasymptotic}.

Theorem \ref{thasymptotic} and Theorem \ref{thasymptotic3} show a qualitative difference between the random effects induced on a wave propagating in long range and in rapidly decorrelating random media in time (see $\eqref{SDCAcor}$).  As previously noted, it has been shown in \cite{bal2} for rapidly decorrelating random media that the field $\phi_e$ does not evolve before the scale $s=1$, so that there is no wave decoherence before $s=1$. In this case the phase and the phase space energy of the wave, which describes the wave decoherence, evolve on the same scale. Let us give a probabilistic interpretation to illustrate the difference of the random effects caused by the long range and rapidly decorrelating random media. The wave decoherence is given by the random operator \eqref{lepsilon}, which after homogenization gives rise to an operator of the form
\begin{equation}\label{levy0}\e^{1-s}\int d\xt{p}\sigma(\xt{p})\Big(\lambda\big(\xt{k}+\frac{\xt{p}}{\e^{s_c}}\big)-\lambda(\xt{k})\Big).\end{equation}
Formally, the momentum diffusion of the Wigner is given by the variations of a stochastic jump process \cite[Chapter 1]{sato} with infinitesimal generator given by \eqref{levy0}, which characterizes its variations. In the case of rapidly decaying correlations $\sigma(\xt{p})\in L^1(\mathbb{R^d})$ so that the variations of the jump process are therefore bounded by $\mathcal{O}(\e^{1-s})$ for all $s_c$, that is why we cannot observe wave decoherence phenomena in rapidly decorrelating random media for any $s_c>0$. The variations of the jump process become significant only for the scaling $s=1$ and $s_c=0$, and that is why we can observe wave decoherence only at this scaling. However, for long-range random media the variations of the jump process can be large and described by
\[\e^{1-s}\int d\xt{p}\frac{a(\xt{p})}{\lvert \xt{p}\rvert^{d+\theta}}\Big(\lambda\big(\xt{k}+\frac{\xt{p}}{\e^{s_c}})-\lambda(\xt{k})\Big)\underset{\e\to 0}{\sim} \e^{1-s-\theta s_c} a(0)\int \frac{d\xt{p}}{\lvert \xt{p}\rvert^{d+\theta}}\Big(\lambda(\xt{k}+\xt{p})-\lambda(\xt{k})\Big). \]
As a result the large variations can balance the small term $\e^{1-s}$ and give rise to significant momentum diffusion. In rapidly decorrelating random media the wave decoherence is only significant for $s_c=0$ and for a sufficiently large propagation distance, while for the long-range random media the wave decoherence occurs first on the large correlation scale and propagates to the smaller ones as the propagation scale increase, up to the scaling $s_c=0$ and $s=1$ (see Figure \ref{schema}). We will see in the next section that the wave decoherence mechanism in the scaling $s=1$ and $s_c=0$ for long-range and rapidly decorrelating random media is exactly the same, but however, this decoherence mechanism has different regularity properties in both cases.

Let us remark that the spatial frequency of the initial condition \eqref{initcond} which is of order $\e^{-(s-s_c)}$ is low compared to the one of the random medium ($\sim \e^{-s}$) on the macroscopic scale $\e^{-s}$. In rapidly deccorelating random media such low frequency sources do not interact with the random medium. However, as we have seen in Theorem \ref{thasymptotic} and Theorem \ref{thasymptotic3} low frequency initial conditions interact strongly with slowly decorrelating random media. This fact can be useful in passive imaging of a target in a slowly decorrelating random medium \cite{garnier2}.

\begin{figure}\begin{center}
$(a)$\includegraphics[scale=0.25]{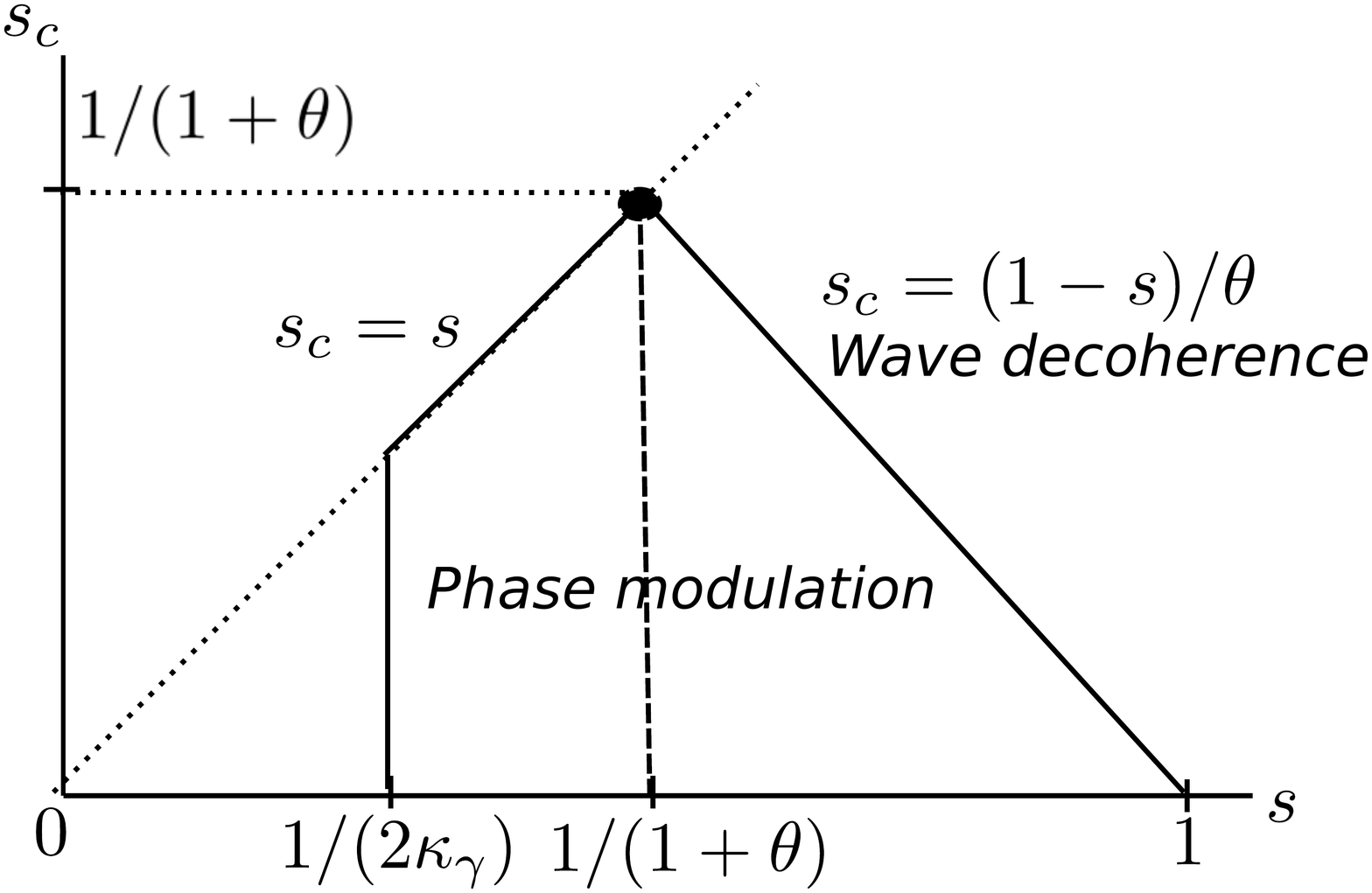}$(b)$\includegraphics[scale=0.25]{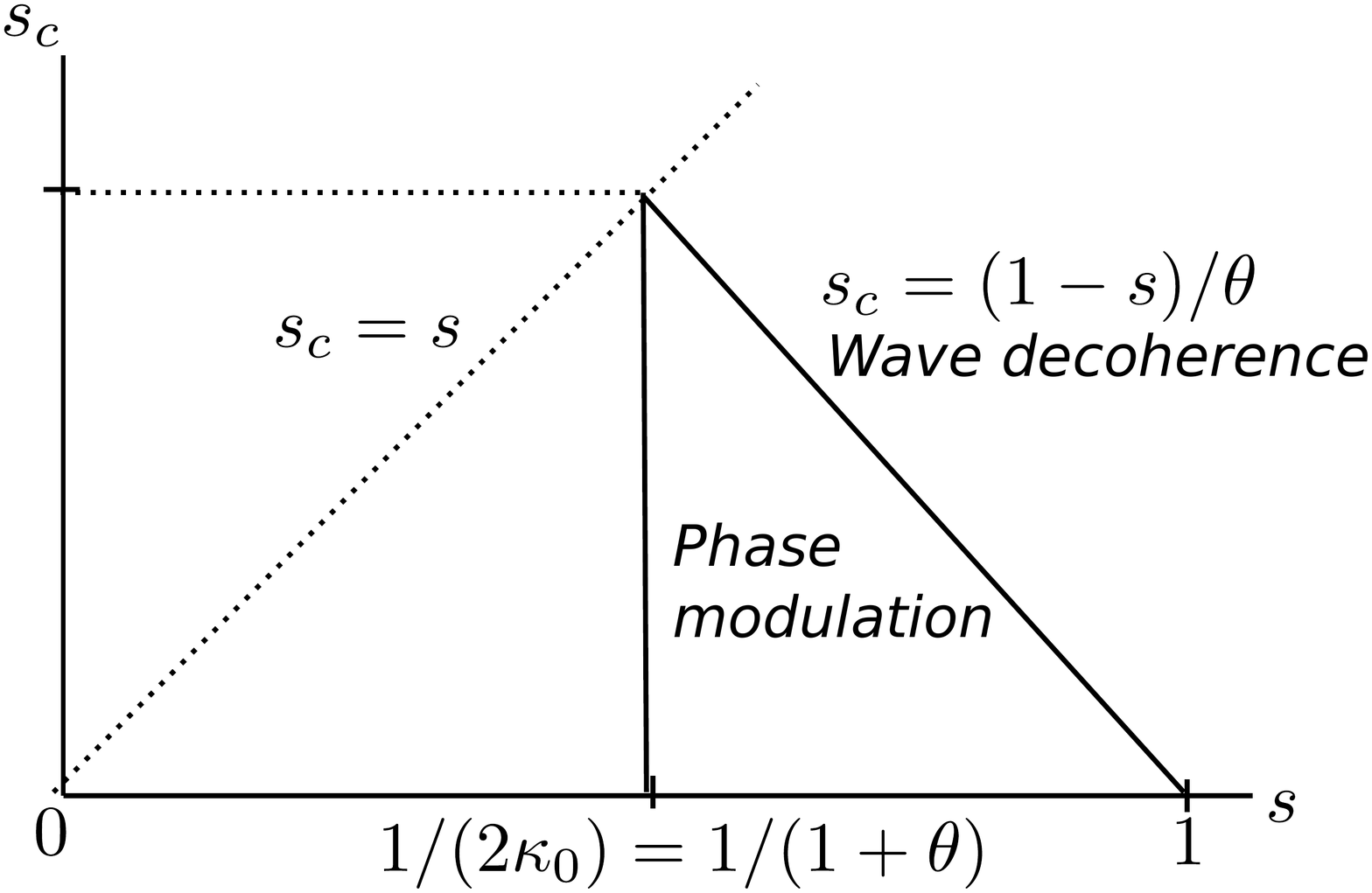}\caption{\label{schema} Schematic representation of the behavior of the wave field $\phi_\e$ solution of \eqref{schrodingereq}. $s$ is the propagation scale parameter and $s_c$ is the spatial correlation parameter. The phase modulation effects appear for $s=1/(2\kappa_\ga)$ for all $s_c\leq s$. Afterward, wave decoherence appears on the low spatial correlation scales first, and then propagates to the higher one according to the formula $s_c=(1-s)/\theta$. In $(a)$ we represent the behavior of the wave field in the case where $\ga>0$ or $\beta<1/2$. The dot represents the transition between the phase modulation effects and the wave decoherence effects (see Theorem \ref{thasymptotic3}). In $(b)$  we represent the behavior of the wave in the case where $\ga=0$ and $\beta=1/2$.}
\end{center}\end{figure}

\section{The Radiative Transport Scaling $s=1$ and $s_c=0$}\label{sections1}

This section describes the evolution of the wave decoherence mechanism in which the momentum diffusion and the dispersion are of order $1$. The following results have been proved in \cite{gomez}, but we state these results in order to provide a complete and self-contained presentation of the wave decoherence mechanism in long-range random media. 

The radiative transport scaling is an important scaling since it is the only one for which wave decoherence happens for rapidly deccorelating random media. This scaling provides also wave decoherence for long-range random media, but the momentum  diffusion is not exactly the fractional diffusion as previously obtained. Even if the decoherence mechanism are the same in long-range and rapidly decorrelating random media this result highlights an important qualitative difference between the two kinds of random media.

In the radiative transfer scaling in addition to a momentum diffusion, now we also have a dispersion term $\xt{k}\cdot \nabla_{\xt{x}}$. Using the notation introduced in Section \ref{wignersection}, we have the following result.

\begin{thm}\label{thasymptotic2}
Let us assume that \eqref{SDCA} holds. For $\ga>0$, the family $(W_\e)_{\e\in(0,1)}$ of Wigner transform defined by \eqref{wignertrans} with $s=1$ and $s_c=0$, solution of the transport equation \eqref{wignereq}, converges in probability on 
$\mathcal{C}([0,+\infty), (\mathcal{B}_{r} ,d_{\mathcal{B}_{r}}))$ as $\e \to 0$ to a limit denoted by $W$. More precisely, for all $T>0$ and for all $\eta>0$,
\[\lim_{\e\to 0}\Pro\left(\sup_{t\in[0,T]}d_{\mathcal{B}_r}(W_\e(t),W(t))>\eta\right)=0,\]
where $W$ is the unique classical solution uniformly bounded in $L^2(\mathbb{R}^{2d})$ of the radiative transfer equation
\begin{equation}\label{radtranseq3}
\partial_t W+\xt{k}\cdot \nabla_{\xt{x}} W=\mathcal{L}W,
\end{equation} 
with $W(0,\xt{x},\xt{k})=W_0(\xt{x},\xt{k})\in L^2(\mathbb{R}^{2d})$ defined by \eqref{w01}. 
Here, $\mathcal{L}$ is defined by
\begin{equation}\label{formgene}
\mathcal{L}\varphi(\xt{k})=\int d\xt{p}\sigma(\xt{p}-\xt{k})\big(\varphi(\xt{p})-\varphi(\xt{k})\big),
\end{equation}
with $\varphi\in\mathcal{C}^\infty (\mathbb{R}^d)$ and where 
\[\sigma(\xt{p})=\frac{2\widehat{R}_0(\xt{p})}{(2\pi)^d\mathfrak{g}(\xt{p})}=\frac{2a(\xt{p})}{(2\pi)^d \lvert \xt{p}\rvert^{d+\theta}},\]
with  $\theta\in(0,1)$. Moreover, $W$ is given by
\[W(t,\xt{x},\xt{k})=\frac{1}{(2\pi)^{2d}}\int d\xt{y}d\xt{q} e^{i(\xt{x}\cdot\xt{y}+\xt{k}\cdot\xt{q})}e^{\int_0^t du \Psi(\xt{q}+u\xt{y})}\widehat{W}_0(\xt{y},\xt{q}+t\xt{y}),\]
where 
\[\Psi(\xt{q})=\int d\xt{p}\sigma(\xt{p})(e^{i\xt{p}\cdot\xt{q}}-1),\]
so that for all $t_0>0$ we have 
\[
W\in \mathcal{C}^0\Big( (0,+\infty), \bigcap_{k\geq 0} H^k (\mathbb{R}^{2d}) \Big)\cap L^\infty \Big([t_0,+\infty),\bigcap _{k\geq 0} H^k (\mathbb{R}^{2d})\Big).
\]
\end{thm}

In Theorem \ref{thasymptotic2} $H^k(\mathbb{R}^{2d})$ stands for the $k$Th. Sobolev space on $\mathbb{R}^{2d}$, and $\widehat{W}_0$ stands for the Fourier transform in both variables $\xt{x}$ and $\xt{k}$. Let us note that the case $\ga=0$, has not been addressed in \cite{gomez}, because it leads to much more difficult algebra than the cases $\ga\in(0,1)$. More precisely,  we show in this case the tightness of the family $(W_\e)_\e$ and show that all the subsequence limits are deterministic weak solutions of the same transport equation \eqref{radtranseq3} with
\[\mathcal{L}\varphi(\xt{k})=\int d\xt{p}\sigma\Big(\xt{p}-\xt{k},\frac{\lvert \xt{k}\rvert^2-\lvert \xt{p}\rvert^2}{2}\Big)\big(\varphi(\xt{p})-\varphi(\xt{k})\big)\quad\text{with}\quad\sigma(\xt{p},\omega)=\frac{2\mathfrak{g}(\xt{p})\widehat{R}_0(\xt{p})}{(2\pi)^d(\mathfrak{g}(\xt{p})^2+\omega^2)}.\]

However, it is difficult to show the weak uniqueness of the limiting transfer equation in the slowly decorrelating case. First, the technique used in the proof of Theorem \ref{thasymptotic2} to show the weak uniqueness leads to very difficult algebra. Second, it should be possible to use the techniques developed in \cite{caffarelli}, but this kind of techniques use a lower and an upper bound of $\sigma$ in terms of $\lvert \xt{k}-\xt{p}\rvert^{-(d+\theta)}$, and we just have an upper bound of this form. Nevertheless, we think that the transport equation obtained in the case $\ga=0$ is still weakly well posed.

As already discussed in Section \ref{sectionint}, we cannot expect a convergence on $L^2(\mathbb{R}^{2d})$ equipped with the strong topology in Theorem  \ref{thasymptotic2} since the conservation relation $\|W_\e(t)\|_{L^2(\mathbb{R}^{2d})}=\|W_{0,\e}\|_{L^2(\mathbb{R}^{2d})}$ cannot be satisfied by the limit $W$. Moreover, the Wigner distribution $W_\e$ is also self-averaging as $\e$ goes to $0$, that is the limit $W$ is not random anymore. This self-averaging phenomenon of the Wigner distribution has already been observed in several studies \cite{bal3, bal,bal4,gomez} and is very useful for applications.

Equation \eqref{radtranseq3} describes the loss of coherence of the wave field $\phi_\e$ solution of the random Schrödinger equation \eqref{schrodingereq}. It describes the dispersion phenomenon through the transport term $\xt{k}\cdot\nabla_\xt{x}$, and the wave decoherence through the nonlocal transfer operator $\mathcal{L}$ defined by \eqref{formgene}. Finally, the transfer coefficient $\sigma(\xt{p}-\xt{k})$ describes the energy transfer between the modes $\xt{k}$ and $\xt{p}$.   

An interesting remark is that the result of Theorem \ref{thasymptotic2} does not depend on whether $\int d\xt{p} \sigma(\xt{p})$ is finite or not. In other words, the radiative transfer equation \eqref{radtranseq3} is valid in the two case, slowly and rapidly decaying correlations in time (see \eqref{SDCAcor}). However, as noted in Theorem \ref{thasymptotic2}, these equations, in the both cases, behave in different ways. As it has been discussed in Section \ref{SDCAsec}, in the case of rapidly decaying correlations in time, that is $\int d\xt{p} \sigma(\xt{p})<+\infty$, the radiative transfer equation \eqref{radtranseq3} has the same properties as in the mixing case addressed in \cite{bal6}. In the case of slowly decaying correlations in time, that is $\int d\xt{p} \sigma(\xt{p})=+\infty$, we observe a regularizing effect \cite[Theorem 4.1]{gomez} of the solutions of  \eqref{radtranseq3} which cannot be observed in the case of rapidly decaying correlations in time. This regularizing property highlights an important qualitative difference between the cases rapidly and slowly decaying correlations, and admits the following probabilistic representation in terms of Lévy processes \cite[Chapter 1]{sato}. According to \cite[Proposition 5.1]{gomez}, we have
\begin{equation}\label{probarep}W(t, \xt{x},\xt{k})=\E\Big[W_0\Big(\xt{x}-t\xt{k}-\int_0^t L_sds,\xt{k}+L_t\Big)\Big],\end{equation}
where $W$ is the unique classical solution of the radiative transport equation \eqref{radtranseq3} with initial datum $W_0\in L^2(\mathbb{R}^{2d})$, and where $(L_t)_{t\geq0}$ is a Lévy process with $L_0=0$ and infinitesimal generator \eqref{formgene}. Let us note that $(L_t)_{t\geq0}$ is a pure jump process with jump measure $\sigma(\xt{p})d\xt{p}$. Therefore, because of the long-range correlation property in time \eqref{SDCAcor} the jump process has infinitely many small jumps with probability one \cite[Theorem 21.3 pp. 136]{sato}. This last property of the symmetric process $(L_t)_{t\geq 0}$ is the key to show the regularizing effect of the radiative transfer equation \eqref{radtranseq3}. In fact, according to \cite[Proposition 1.1]{picard}, the Levy process $(L_t)_{t\geq0}$ has a smooth bounded density, which permits to obtain smoothness on the variable $\xt{k}$. Moreover, the transport term in \eqref{radtranseq3} permits to transfer the smoothness from the variable $\xt{k}$ to the variable $\xt{x}$. However, this regularizing effect cannot be observed in the case rapidly decaying random media in time. In fact, in this case we have the same probabilistic representation \eqref{probarep} but where the Lévy process $(L_t)_{t\geq0}$ has a bounded jump measure $\sigma(\xt{p})d\xt{p}$ (see  \eqref{SDCAcor}). In this context, there is a positive probability that on any finite time interval $(L_t)_{t\geq0}$ does not jump \cite[Theorem 21.3 pp. 136]{sato}. Consequently, in this context there is not enough jump to induce regularization on the Wigner transform. 

Let us note that the momentum diffusion in \eqref{radtranseq3} is not exactly the same diffusion mechanism as the one obtained in Section \ref{sectionint}, but they are very closed. In fact the momentum diffusion given in Theorem \ref{thasymptotic} is described in terms of a fractional Laplacian, while in the radiative transfer regime the momentum diffusion is described in terms of a nonlocal operator which is not exactly a fractional Laplacian. However, this two diffusion mechanisms are anomalous diffusions since they lead to damping terms obeying to a power law with exponent $\theta\in(0,1)$. We have to wait for a long time of propagation in the radiative transfer regime to observe again the momentum diffusion given by a fractional Laplacian. This approximation in the radiative transfer scaling is proved in \cite[Theorem 5.1]{gomez}

\section*{Conclusion}

In this paper we have studied the different behaviors happening on a wave propagating in a random media with long-range correlation properties. We have exhibited three different behaviors over a range of scales given by Theorem \ref{thphase}, Theorem \ref{thasymptotic}, Theorem \ref{thasymptotic3}, and Theorem \ref{thasymptotic2}. These asymptotic behaviors differ strongly from those obtain with the random Schrödinger equation with rapidly decaying correlations \cite{bal0,bal6,bal2}, for which all the random effects appear on the wave on the same propagation scale $\e^{-1}$ ($s=1$). In the context of long-range correlations, the effects of the randomness appear progressively according the scale of propagation. We have seen in Theorem  \ref{thphase} that the random perturbations induce a phase modulation in term of fractional Brownian motion on the wave field itself. This random phase modulation begins to oscillate very fast up to break the wave coherence. The first wave decoherence mechanism is described by a random phase modulation on the Wigner transform (Theorem \ref{thasymptotic3}) on the large spatial scale.  Afterward, this phase modulation on the Wigner transform begins also to oscillate very fast up to average out and then gives a deterministic wave decoherence mechanism. The wave decoherence first happens on the large spatial scales and then propagates to smaller one as the propagation distance increases (Theorem \ref{thasymptotic}, Theorem \ref{thasymptotic3}, Theorem \ref{thasymptotic2}, and see Figure \ref{schema}). The wave decoherence mechanism is described in term of an anomalous momentum diffusion (Theorem \ref{thasymptotic} and Theorem \ref{thasymptotic2}) since it obeys to a power law with exponent lying in $(0,1)$. Theorem \ref{thasymptotic} shows that low frequency initial conditions interact strongly with slowly decorrelating random media. This result can be useful in passive imaging of a target in a long-range random medium \cite{garnier2}.

\section{Proof of Proposition \ref{propmar}}\label{proofpropmar}

First, to prove \eqref{markovesp} it suffices to show, for all $n\geq 1$ and $0\leq t_1\leq \dots\leq t_{n+1}$, that
\[\mathbb{E}\big[ \widehat{V}(t_{n+1},\cdot) \vert \widehat{V}(t_{1},\cdot),\dots ,\widehat{V}(t_{n},\cdot)  \big]=e^{-\mathfrak{g}(\xt{p})(t_{n+1}-t_n)}\widehat{V}(t_n,\cdot),\]
on $\mathcal{H}'_m$. For $n=1$ and $0\leq t_1\leq t_2$, we write
\[\widehat{V}(t_{2},\cdot)=e^{-\mathfrak{g}(\xt{p})(t_{2}-t_1)}\widehat{V}(t_{1},\cdot) + Y,\]
where $Y$ and $\widehat{V}(t_{1},\cdot)$ are independent, since they are zero-mean Gaussian variables and
 \[\begin{split} 
 \E[Y(\varphi)\widehat{V}(t_1)(\psi)]&= \E[\widehat{V}(t_2)(\varphi)\widehat{V}(t_1)(\psi)]- \E[\widehat{V}(t_1)(\varphi_{t_2-t_1})\widehat{V}(t_1)(\psi)]\\
 &=\int m(d\xt{p})\,\varphi(\xt{p})\overline{\psi(\xt{p})}(e^{-\mathfrak{g}(\xt{p})(t_{2}-t_1)}-e^{-\mathfrak{g}(\xt{p})(t_{2}-t_1)} )\\
 &=0,
 \end{split} \]
 where $\varphi_{s}(\xt{p})=e^{-\mathfrak{g}(\xt{p})s}\varphi(\xt{p})$
 and for all $(\varphi,\psi)\in\mathcal{H}_m\times\mathcal{H}_m$. As a result, we have
 \[\mathbb{E}\big[ \widehat{V}(t_{2},\cdot) \vert \widehat{V}(t_{1},\cdot)\big]=e^{-\mathfrak{g}(\xt{p})(t_{2}-t_1)}\widehat{V}(t_1,\cdot).\]
Now, let us fix $n\geq 2$ and assume that for all family $(s_j)_{j\in\{1,\dots,n\}}$ such that $0\leq s_1\leq \dots\leq s_{n}$
\[  \mathbb{E}\big[ \widehat{V}(s_{n},\cdot) \vert \widehat{V}(s_{1},\cdot),\dots ,\widehat{V}(s_{n-1},\cdot)  \big]=e^{-\mathfrak{g}(\xt{p})(s_{n}-s_{n-1})}\widehat{V}(s_{n-1},\cdot), \]
Then, we write
\[\widehat{V}(t_{n+1},\cdot)=e^{-\mathfrak{g}(\xt{p})(t_{n+1}-t_n)}\widehat{V}(t_{n},\cdot) + Y,\]
where $Y$ and $\widehat{V}(t_{n},\cdot)$ are independent as explained above, so that 
\[\begin{split}
\mathbb{E}\big[ \widehat{V}(t_{n+1},\cdot) \vert \widehat{V}(t_{1},\cdot),\dots ,\widehat{V}(t_{n},\cdot)  \big]&= e^{-\mathfrak{g}(\xt{p})(t_{n+1}-t_n)}\widehat{V}(t_n,\cdot)+\mathbb{E}\big[ Y  \vert \widehat{V}(t_{1},\cdot),\dots ,\widehat{V}(t_{n-1},\cdot) \big] \\
&=e^{-\mathfrak{g}(\xt{p})(t_{n+1}-t_n)}\widehat{V}(t_n,\cdot)+\mathbb{E}\big[  \widehat{V}(t_{n+1},\cdot)  \vert \widehat{V}(t_{1},\cdot),\dots ,\widehat{V}(t_{n-1},\cdot) \big]\\
&-e^{-\mathfrak{g}(\xt{p})(t_{n+1}-t_n)}\mathbb{E}\big[  \widehat{V}(t_{n},\cdot)  \vert \widehat{V}(t_{1},\cdot),\dots ,\widehat{V}(t_{n-1},\cdot) \big]\\
&= e^{-\mathfrak{g}(\xt{p})(t_{n+1}-t_n)}\widehat{V}(t_n,\cdot)\\
&+(e^{-\mathfrak{g}(\xt{p})(t_{n+1}-t_{n-1})}-e^{-\mathfrak{g}(\xt{p})(t_{n+1}-t_{n})}e^{-\mathfrak{g}(\xt{p})(t_{n}-t_{n-1})})\widehat{V}(t_{n-1},\cdot),
\end{split}\]
which concludes the proof of \eqref{markovesp} by induction. Second, to prove \eqref{markovvar} it suffices to show, for all $n\geq 1$, $0\leq t_1\leq \dots\leq t_{n+1}\leq \tilde{t}_{n+1}$ and $(\varphi,\psi)\in\mathcal{H}_m\times \mathcal{H}_m$, that
\[\begin{split}\mathbb{E}\big[ \widehat{V}(\tilde{t}_{n+1})(\varphi) &\widehat{V}(t_{n+1})(\psi) \vert \widehat{V}(t_{1},\cdot),\dots ,\widehat{V}(t_{n},\cdot)  \big]=\mathbb{E}\big[ \widehat{V}(\tilde{t}_{n+1})(\varphi) \vert  \widehat{V}(t_{1},\cdot),\dots ,\widehat{V}(t_{n},\cdot) \big]\\
&\hspace{5cm}\times\mathbb{E}\big[ \widehat{V}(t_{n+1})(\psi) \vert  \widehat{V}(t_{1},\cdot),\dots ,\widehat{V}(t_{n},\cdot) \big]\\
&+\int m(d\xt{p})\varphi(\xt{p})\overline{\psi(\xt{p})}(e^{-\mathfrak{g}(\xt{p})(\tilde{t}_{n+1}-t_{n+1})}-e^{-\mathfrak{g}(\xt{p})(\tilde{t}_{n+1}-t_{n})}e^{-\mathfrak{g}(\xt{p})(t_{n+1}-t_{n})}).
\end{split}\]
For $n=1$ and $0\leq t_1\leq t_2\leq\tilde{t}_2$, we write
\[\widehat{V}(\tilde{t}_{2},\cdot)=e^{-\mathfrak{g}(\xt{p})(\tilde{t}_{2}-t_1)}\widehat{V}(t_{1},\cdot) + \tilde{Y}\quad \text{and}\quad\widehat{V}(t_{2},\cdot)=e^{-\mathfrak{g}(\xt{p})(t_{2}-t_1)}\widehat{V}(t_{1},\cdot) + Y,\]
where $\tilde{Y}$ and $\widehat{V}(t_{1},\cdot)$ are independent as well as $Y$ and $\widehat{V}(t_{1},\cdot)$, so that
\[\mathbb{E}\big[ \widehat{V}(\tilde{t}_{2})(\varphi) \widehat{V}(t_{2})(\psi) \vert \widehat{V}(t_{1},\cdot) \big]=\mathbb{E}\big[ \widehat{V}(\tilde{t}_{2})(\varphi) \vert  \widehat{V}(t_{1},\cdot) \big]\mathbb{E}\big[ \widehat{V}(t_{2})(\psi) \vert  \widehat{V}(t_{1},\cdot) \big]+\E[\tilde{Y}(\varphi)Y(\psi)],\]
with
\[
\E[\tilde{Y}(\varphi)Y(\psi)]= \int m(d\xt{p})\varphi(\xt{p})\overline{\psi(\xt{p})}(e^{-\mathfrak{g}(\xt{p})(\tilde{t}_{2}-t_{2})}-e^{-\mathfrak{g}(\xt{p})(\tilde{t}_{2}-t_{1})}e^{-\mathfrak{g}(\xt{p})(t_{2}-t_{1})}).
\]
Now, let us fix $n\geq 2$ and assume that for all family $(s_j)_{j\in\{1,\dots,n\}}$ such that $0\leq s_1\leq \dots\leq s_{n}\leq \tilde{s}_{n}$ and for all $(\varphi,\psi)\in\mathcal{H}_m\times \mathcal{H}_m$,
\[\begin{split}\mathbb{E}\big[ \widehat{V}(\tilde{s}_{n})(\varphi) &\widehat{V}(s_{n})(\psi) \vert \widehat{V}(s_{1},\cdot),\dots ,\widehat{V}(s_{n-1},\cdot)  \big]=\mathbb{E}\big[ \widehat{V}(\tilde{s}_{n})(\varphi) \vert  \widehat{V}(s_{1},\cdot),\dots ,\widehat{V}(s_{n-1},\cdot) \big]\\
&\hspace{5cm}\times\mathbb{E}\big[ \widehat{V}(s_{n})(\psi) \vert  \widehat{V}(s_{1},\cdot),\dots ,\widehat{V}(s_{n-1},\cdot) \big]\\
&+\int m(d\xt{p})\varphi(\xt{p})\overline{\psi(\xt{p})}(e^{-\mathfrak{g}(\xt{p})(\tilde{s}_{n}-s_{n})}-e^{-\mathfrak{g}(\xt{p})(\tilde{s}_{n}-s_{n-1})}e^{-\mathfrak{g}(\xt{p})(s_{n}-s_{n-1})}).
\end{split}\]
Consequently, writing 
\[\widehat{V}(\tilde{t}_{n+1},\cdot)=e^{-\mathfrak{g}(\xt{p})(\tilde{t}_{n+1}-t_n)}\widehat{V}(t_{n},\cdot) + \tilde{Y}\quad \text{and}\quad\widehat{V}(t_{n+1},\cdot)=e^{-\mathfrak{g}(\xt{p})(t_{n+1}-t_n)}\widehat{V}(t_{n},\cdot) + Y,\]
where $\tilde{Y}$ and $\widehat{V}(t_{n},\cdot)$ are independent as well as $Y$ and $\widehat{V}(t_{n},\cdot)$, we have
\[\begin{split}
\mathbb{E}\big[ \widehat{V}(\tilde{t}_{n+1})(\varphi) \widehat{V}(t_{n+1})(\psi) \vert  \widehat{V}(t_{1},\cdot),\dots ,\widehat{V}(t_{n},\cdot) \big]&=\mathbb{E}\big[ \widehat{V}(\tilde{t}_{n+1})(\varphi) \vert  \widehat{V}(t_{1},\cdot),\dots ,\widehat{V}(t_{n},\cdot) \big]\\
&\times\mathbb{E}\big[ \widehat{V}(t_{n+1})(\psi) \vert  \widehat{V}(t_{1},\cdot),\dots ,\widehat{V}(t_{n},\cdot) \big]\\
&+\E[\tilde{Y}(\varphi)Y(\psi)\vert  \widehat{V}(t_{1},\cdot),\dots ,\widehat{V}(t_{n-1},\cdot) ]
\end{split}\]
since
\[\E[Y\vert  \widehat{V}(t_{1},\cdot),\dots ,\widehat{V}(t_{n-1},\cdot) ] = (e^{-\mathfrak{g}(\xt{p})(\tilde{t}_{n+1}-t_{n-1})}-e^{-\mathfrak{g}(\xt{p})(\tilde{t}_{n+1}-t_n)}e^{-\mathfrak{g}(\xt{p})(\tilde{t}_{n}-t_{n-1})})\widehat{V}(t_{n},\cdot)=0,\]
and in the same way $ \E[\tilde{Y}\vert  \widehat{V}(t_{1},\cdot),\dots ,\widehat{V}(t_{n-1},\cdot) ]=0$. Finally, we have
\[\begin{split}
\E[\tilde{Y}(\varphi)Y(\psi)\vert  \widehat{V}(t_{1},\cdot),\dots ,\widehat{V}(t_{n-1},\cdot) ]&=\E[ \widehat{V}(\tilde{t}_{n+1})(\varphi) \widehat{V}(t_{n+1})(\psi)\vert  \widehat{V}(t_{1},\cdot),\dots ,\widehat{V}(t_{n-1},\cdot) ]\\
& -\E[ \widehat{V}(\tilde{t}_{n+1})(\varphi) \widehat{V}(t_{n})(\psi_{t_{n+1}-t_n})\vert  \widehat{V}(t_{1},\cdot),\dots ,\widehat{V}(t_{n-1},\cdot) ]\\
&-\E[ \widehat{V}(t_{n})(\varphi_{\tilde{t}_{n+1}-t_n}) \widehat{V}(t_{n+1})(\psi)\vert  \widehat{V}(t_{1},\cdot),\dots ,\widehat{V}(t_{n-1},\cdot) ]\\
&+\E[ \widehat{V}(t_{n})(\varphi_{\tilde{t}_{n+1}-t_n}) \widehat{V}(t_{n})(\psi_{t_{n+1}-t_n})\vert  \widehat{V}(t_{1},\cdot),\dots ,\widehat{V}(t_{n-1},\cdot) ],
\end{split}\]
where $\varphi_{s}(\xt{p})=e^{-\mathfrak{g}(\xt{p})s}\varphi(\xt{p})$ and $\psi_{s}(\xt{p})=e^{-\mathfrak{g}(\xt{p})s}\psi(\xt{p})$.
As a result, we have
\[\E[\tilde{Y}(\varphi)Y(\psi)\vert  \widehat{V}(t_{1},\cdot),\dots ,\widehat{V}(t_{n-1},\cdot) ]=P_1+P_2,\]
where
\[\begin{split}P_1&=\widehat{V}(t_{n-1})(\varphi_{\tilde{t}_{n+1}-t_{n-1}}) \widehat{V}(t_{n-1})(\psi_{t_{n+1}-t_{n-1}})\\
&-\widehat{V}(t_{n-1})(\varphi_{\tilde{t}_{n+1}-t_{n-1}}) \widehat{V}(t_{n-1})(\psi_{t_{n+1}-t_n+t_n-t_{n-1}})\\
&-\widehat{V}(t_{n-1})(\varphi_{\tilde{t}_{n+1}-t_n+t_n-t_{n-1}}) \widehat{V}(t_{n-1})(\psi_{t_{n+1}-t_{n-1}})\\
&+\widehat{V}(t_{n-1})(\varphi_{\tilde{t}_{n+1}-t_n+t_n-t_{n-1}}) \widehat{V}(t_{n-1})(\psi_{t_{n+1}-t_n+t_n-t_{n-1}})\\
&=0,
\end{split}\]
and where
\[\begin{split}
P_2=&\int m(d\xt{p})\varphi(\xt{p})\overline{\psi(\xt{p})}[e^{-\mathfrak{g}(\xt{p})(\tilde{t}_{n+1}-t_{n+1})}-e^{-\mathfrak{g}(\xt{p})(\tilde{t}_{n+1}-t_{n-1})}e^{-\mathfrak{g}(\xt{p})(t_{n+1}-t_{n-1})}\\
&-e^{-\mathfrak{g}(\xt{p})(t_{n+1}-t_{n})}(e^{-\mathfrak{g}(\xt{p})(\tilde{t}_{n+1}-t_{n})}-e^{-\mathfrak{g}(\xt{p})(\tilde{t}_{n+1}-t_{n-1})}e^{-\mathfrak{g}(\xt{p})(t_{n}-t_{n-1})})\\
&-e^{-\mathfrak{g}(\xt{p})(\tilde{t}_{n+1}-t_{n})}(e^{-\mathfrak{g}(\xt{p})(t_{n+1}-t_{n})}-e^{-\mathfrak{g}(\xt{p})(t_{n}-t_{n-1})}e^{-\mathfrak{g}(\xt{p})(t_{n+1}-t_{n-1})})\\
&+e^{-\mathfrak{g}(\xt{p})(\tilde{t}_{n+1}-t_{n})}e^{-\mathfrak{g}(\xt{p})(t_{n+1}-t_{n})}(1-e^{-2\mathfrak{g}(\xt{p})(t_n-t_{n-1})})]\\
&=\int m(d\xt{p})\varphi(\xt{p})\overline{\psi(\xt{p})}[e^{-\mathfrak{g}(\xt{p})(\tilde{t}_{n+1}-t_{n+1})}-e^{-\mathfrak{g}(\xt{p})(\tilde{t}_{n+1}-t_{n})}e^{-\mathfrak{g}(\xt{p})(t_{n+1}-t_{n})}],
\end{split}\]
which concludes the proof of \eqref{markovvar} by induction.

\section{Proof of Theorem \ref{thasymptotic}}\label{proof}

The proof of Theorem \ref{thasymptotic} is based on the perturbed-test-function approach and follows the techniques of the proof of \cite[Theorem 2.2]{gomez}. Using the notion of pseudogenerator, we prove first the tightness and then characterize all the subsequence limits.

\subsection{Pseudogenerator}\label{pseudogene}

We recall the techniques developed by Kurtz and Kushner \cite[Section 2.2 pp. 38]{kushner}. Here, let us consider the filtration $\mathcal{F}^\e _t =\mathcal{F}_{t/\e}$ where $(\mathcal{F}_{t})$ is defined by \eqref{filtration}. Let $\mathcal{M}^\e $ be the set of all $\mathcal{F}^\e$-measurable functions $f(t)$, adapted to the filtration $(\mathcal{F}^\e _t)$, for which $\sup_{t\leq T} \mathbb{E}\left[\lvert f(t) \rvert \right] <+\infty $ and where $T>0$ is fixed.  The $p-\lim$ and the pseudogenerator are defined as follows. Let $f$ and $f^\delta$ in $\mathcal{M}^\e $ for all $\delta>0$. We say that $f=p-\lim_\delta f^\delta$ if
\[ \sup_{t, \delta }\mathbb{E}[\lvert f^\delta(t)\rvert]<+\infty\quad \text{and}\quad \lim_{\delta\rightarrow 0}\mathbb{E}[\lvert f^\delta (t) -f(t)\rvert]=0 \quad \forall t\geq 0.\]
We say that $f\in \mathcal{D}\left(\mathcal{A}^\e\right)$ the domain of $\mathcal{A}^\e$ and $\mathcal{A}^\e f=g$ if $f$ and $g$ are in $\mathcal{M}^\e$ and 
\begin{equation*}
p-\lim_{\delta \to 0} \left[ \frac{\mathbb{E}^\e _t [f(t+\delta)]-f(t)}{\delta}-g(t) \right]=0,
\end{equation*}
where $\mathbb{E}^\e _t$ is the conditional expectation given $\mathcal{F}^\e _t$. A useful result about the pseudogenerator $\mathcal{A}^\e$ is given by the following theorem.
\begin{thm}\label{martingale}
Let $f\in \mathcal{D}\left(\mathcal{A}^\e\right)$. Then
\begin{equation*}
M_f ^\e (t)=f(t)-f(0)-\int _0 ^t  \mathcal{A}^\e f(u)du
\end{equation*}
is an $( \mathcal{F}^\e _t )$-martingale.
\end{thm}

\subsection{Tightness}\label{tightnesssec}

In this section we prove the tightness of the family $(W_\e)_{\e \in (0,1)}$ on the polish space $\mathcal{C}([0,+\infty),(\mathcal{B}_{r},d_{\mathcal{B}_{r}}))$ according to the following criteria. Following  \cite[Section 7 pp. 80]{billingsley} and \cite[Theorem 4.9 pp. 62, Theorem 4.10 pp. 63]{Karatzas}, we have the following version of the Arzelà-Ascoli theorem for processes with values in a complete separable metric space. 
\begin{thm}
A set $B\subset \mathcal{C}([0,+\infty),(\mathcal{B}_{r},d_{\mathcal{B}_{r}} ))$ has a compact closure if and only if
\[\sup_{g\in B}d_{\mathcal{B}_r}( g(0),0) <+\infty,\quad\text{and}\quad \forall T>0, \quad \lim_{\eta \to 0}\sup_{g \in B} m_T (g, \eta)=0,  \]
with 
\[m_T (g,\eta)=\sup_{\substack{(s,t)\in [0,T]^2\\ \lvert t-s \rvert \leq\eta}} d_{\mathcal{B}_{r}}( g(s) ,g(t)). \] 
\end{thm}
From this result, we obtain the classical tightness criterion. 
\begin{thm}
A family of probability measure $\big(\mathbb{P}^\e\big)_{\e\in(0,1)}$ on $\mathcal{C}([0,+\infty),(\mathcal{B}_{r},d_{\mathcal{B}_{r}} ))$ is tight if and only if
\[\forall\eta'>0,\quad\lim_{M\to 0}\sup_{\e\in(0,1)}\mathbb{P}^\e \big(g\,; d_{\mathcal{B}_r}( g(0),0) >M\big)=0,\]
and
\[\forall T>0, \eta'>0 \quad \lim_{\eta \to 0}\sup_{\e \in (0,1) }\mathbb{P}^\e \big(g\,; \,\,m_T (g, \eta)>\eta'\big)=0.\]
\end{thm}
From the definition \eqref{defmetric} of the metric $d_{\mathcal{B}_{r}}$, the tightness criterion becomes the following.
\begin{thm}\label{crit}
A family of processes $(X^\e)_{\e\in(0,1)}$ is tight on $\mathcal{C}([0,+\infty),(\mathcal{B}_{r},d_{\mathcal{B}_{r}} ))$ if and only if the process $\big(\big<X^\e,\lambda\big>_{L^2(\mathbb{R}^{2d})}\big)_{\e\in(0,1)}$ is tight on $\mathcal{C}([0,+\infty),\mathbb{R})$ for all $\lambda \in L^2(\mathbb{R}^{2d})$.
\end{thm}
This last theorem looks like the tightness criteria of Mitoma and Fouque \cite{mitoma,fouque}. For any $\lambda \in L^2(\mathbb{R}^{2d})$, we set 
\[W_{\e,\lambda}(t)=\big<W_{\e}(t),\lambda\big>_{L^2(\mathbb{R}^{2d})}.\]
According to Theorem \ref{crit}, the family $(W_{\e})_\e$ is tight on $\mathcal{C}([0,+\infty), (\mathcal{B}_{r},d_{\mathcal{B}_{r}} ))$ if and only if the family $(W_{\e,\lambda})_\e$ is tight on $\mathcal{C}([0,+\infty),\mathbb{R} )$,  for all $\lambda \in L^2(\mathbb{R}^{2d})$. Furthermore, $(W_\e)_{\e}$ is a family of continuous processes. Then, according to \cite[Theorem 13.4 pp. 142]{billingsley}, it suffices to prove that for all $\lambda \in L^2(\mathbb{R}^{2d})$, the family $(W_{\e,\lambda})_{\e}$ is tight on $\mathcal{D}([0,+\infty),\mathbb{R} )$, which is the set of cad-lag functions with values in $\mathbb{R}$. Finally, using that the process $W_\e$ takes its values in $\mathcal{B}_r$ and that the set of all smooth functions with compact support in $\mathbb{R}^{2d}$, denoted by $\mathcal{C}^{\infty}_0(\mathbb{R}^{2d})$, is dense in $L^2(\mathbb{R}^{2d})$, it is sufficient to show that $(W_{\e,\lambda})_{\e}$ is tight on $\mathcal{D}([0,+\infty),\mathbb{R} )$ for all $\lambda \in \mathcal{C}^{\infty}_0(\mathbb{R}^{2d})$.

\begin{prop}\label{tightness}
For all $\lambda \in\mathcal{C}^{\infty}_0(\mathbb{R}^{2d})$, the family $(W_{\e,\lambda})_{\e\in(0,1)}$ is tight on $\mathcal{D}\left([0,+\infty),\mathbb{R} \right)$.
\end{prop}

\begin{proof}[of Proposition \ref{tightness}]

The proof of Proposition \ref{tightness} consists in applying \cite[Theorem 4 pp. 48]{kushner} which is based on the perturbed-test-function technique . Throughout the proof Proposition \ref{tightness}, let $\lambda \in \mathcal{C}^{\infty}_0(\mathbb{R}^{2d})$, $f$ be a bounded smooth function, and $f_0 ^\e (t)=f(W_{\e,\lambda} (t))$. The pseudogenerator  
\begin{equation}\label{Aef0}
\mathcal{A}^\e f_0 ^\e (t)= f'(W_{\e,\lambda} (t))\Big[\e^{s_c}W_{\e,\lambda_1}(t)+\e^{(1-\ga)/2-s}\big<W_\e(t),\mathcal{L}_\e\lambda(t)\big>_{L^2(\mathbb{R}^{2d})}\Big],
\end{equation}
where $\mathcal{L}_\e\lambda(t,\xt{x},\xt{k})$ is defined by \eqref{lepsilon}, and 
\[
\lambda_1(\xt{x},\xt{k})=\xt{k}\cdot\nabla_{\xt{x}}\lambda(\xt{x},\xt{k}),
\]
is well defined since 
\[\sup_{t\geq0}\E\big[\|\mathcal{L}_\e\lambda(t)\|^2_{L^2(\mathbb{R}^{2d})}\big]<+\infty.\]

Let us remark that \eqref{Aef0} blows up as $\e$ goes to $0$ since $s>1/2$. The goal of the perturbed-test-function method is to modify the test function $f^\e_0$ using small corrector, so that the pseudogenerator applied to the test function with its corrector does not blow up anymore (see Lemma \ref{A1}). Let us introduce the first corrector 
\[
\begin{split}
f^\e _1 (t)=&\e^{(1-\ga)/2-s} f'(W_{\e,\lambda} (t))\int d\xt{x}d\xt{k}W_\e(t,\xt{x},\xt{k})\\
&\hspace{0.5cm}\times\int_{t}^{+\infty} \mathbb{E}^\e _t\Big[ \int\frac{\widehat{V}\big(\frac{u}{\e^{s+\ga}},d\xt{p}\big)}{(2\pi)^d i}e^{i(u-t)\xt{p}\cdot \xt{k}/\e^s}e^{i\xt{p}\cdot\xt{x}/\e^s}\times\Big(\lambda\big(\xt{x},\xt{k}-\frac{\xt{p}}{2\e^{s_c}}\big)-\lambda\big(\xt{x},\xt{k}+\frac{\xt{p}}{2\e^{s_c}}\big)\Big)\Big] du.
\end{split}\]
To be able to apply \cite[Theorem 4 pp. 48]{kushner} and then prove Proposition \ref{tightness}, we have to show the two following lemmas. 
\begin{lem}\label{bound2}  $\forall T>0$, and $\eta>0$
\[\lim_{\e} \Pro\Big(\sup_{0\leq t \leq T} \lvert f^\e _1 (t)\rvert>\eta\Big) =0,\quad \text{and} \quad\lim_\e\sup_{t\geq 0}\mathbb{E}\left[\lvert f^\e _1 (t) \rvert \right]=0.\]
\end{lem}
\begin{lem}\label{A1}
$\forall T>0$, $\big\{\mathcal{A}^\e \big(f^\e _0 +f^\e _1\big)(t), \e \in(0,1), 0\leq t\leq T\big\}$ is uniformly integrable.
\end{lem}
The remaining of this section consists in proving Lemma \ref{bound2} and Lemma \ref{A1}.

\begin{proof}[of Lemma \ref{bound2}]

Using \eqref{markovesp}, we have
\[f^{\e}_1(t)= \e^{\frac{1+ \ga}{2}}f'(W_{\e,\lambda} (t))\big<W_\e(t),\mathcal{L}_{1,\e}\lambda(t)\big>_{L^2(\mathbb{R}^{2d})}\]
with
\[
\mathcal{L}_{1,\e}\lambda(t,\xt{x},\xt{k})=\frac{1}{(2\pi)^d i}\int\frac{\widehat{V}\big(\frac{t}{\e^{s+\ga}},d\xt{p}\big)}{\mathfrak{g}(\xt{p})-i\e^{\ga}\xt{k}\cdot\xt{p}}e^{i\xt{p}\cdot\xt{x}/\e^s}\left(\lambda\big(\xt{x},\xt{k}-\frac{\xt{p}}{2\e^{s_c}}\big)-\lambda\big(\xt{x},\xt{k}+\frac{\xt{p}}{2\e^{s_c}}\big)\right).
\]
Let us note that the proof of Lemma \ref{bound2} is a direct consequence of the following lemma
\begin{lem}\label{bound3} We have
\[\lim_\e \e^{1+\ga} \E\Big[\sup_{t\in[0,T]}\|\mathcal{L}_{1,\e}\lambda(t)\|^2_{L^2(\mathbb{R}^{2d})} \Big]=0.\]
\end{lem}

\begin{proof}[of Lemma \ref{bound3}]

First,
\[ \|\mathcal{L}_{1,\e}\lambda(t)\|^2_{L^2(\mathbb{R}^{2d})}=\frac{1}{(2\pi)^{2d}} \int d\xt{x}d\xt{k} \Big\lvert \int \frac{\widehat{V}\big(\frac{t}{\e^{s+\ga}},d\xt{p}\big)}{\mathfrak{g}(\xt{p})-i\e^{\ga}\xt{k}\cdot\xt{p}}e^{i\xt{p}\cdot\xt{x}/\e^s} \Big(\lambda\big(\xt{x},\xt{k}-\frac{\xt{p}}{2\e^{s_c}}\big)-\lambda\big(\xt{x},\xt{k}+\frac{\xt{p}}{2\e^{s_c}}\big)\Big)  \Big\rvert^2. \]
Let us fixe $\xt{x}$, $\xt{k}$, and $u$, and let
\[\phi_{\lambda,\xt{x},\xt{k},u}(\xt{p})=\frac{1}{i}\frac{e^{i\xt{p}\cdot\xt{x}/\e}}{\mathfrak{g}(\xt{p})-i\e^{\ga}\xt{k}\cdot\xt{p}}\Big(\lambda\big(\xt{x},\xt{k}-\frac{\xt{p}}{2\e^{s_c}}\big)-\lambda\big(\xt{x},\xt{k}+\frac{\xt{p}}{2\e^{s_c}}\big)\Big).\]
According to \eqref{theta} and since $\lambda$ is a smooth function, we have $\phi_{\lambda,\xt{x},\xt{k},u}\in\mathcal{H}_m$. Consequently, $\tilde{V}=\big<\widehat{V},\phi_{\lambda,\xt{x},\xt{k},u}\big>_{\mathcal{H}'_m,\mathcal{H}_m}$ is real-valued zero-mean Gaussian process with a pseudo-metric $M$ on $[0,T]$ given by 
\[M(t_1,t_2)=\mathbb{E}\Big[\Big(\tilde{V}\Big(\frac{t_1}{\e^{s+\ga}}\Big) -\tilde{V}\Big(\frac{t_2}{\e^{s+\ga}}\Big)\Big)^2\Big]^{1/2}.\]
Then, for all $(t_1,t_2)\in[0,T]^2$
\begin{equation}\begin{split}\label{metre}
M^2(t_1,t_2)&= \E\Big[\tilde{V}^2\Big(\frac{t_1}{\e^{s+\ga}}\Big)\Big] +\E\Big[\tilde{V}^2\Big(\frac{t_2}{\e^{s+\ga}}\Big)\Big]- \E\Big[\tilde{V}\Big(\frac{t_1}{\e^{s+\ga}}\Big)\tilde{V}\Big(\frac{t_2}{\e^{s+\ga}}\Big)\Big]\\
&= \int m(d\xt{p}) \frac{2(1-e^{-\mathfrak{g}(\xt{p})\lvert t_1-t_2\rvert/\e^{s+\ga}})}{\mathfrak{g}^2(\xt{p})+\e^{2\ga}(\xt{k}\cdot\xt{p})^2}\Big(\lambda\big(\xt{x},\xt{k}-\frac{\xt{p}}{2\e^{s_c}}\big)-\lambda\big(\xt{x},\xt{k}+\frac{\xt{p}}{2\e^{s_c}}\big)\Big)^2\\
&\leq 2\frac{\lvert t_1-t_2\rvert}{\e^{s+\ga}} \int d\xt{p} \frac{a(\xt{p})}{\lvert \xt{p}\rvert^{d+\theta} }\Big(\lambda\big(\xt{x},\xt{k}-\frac{\xt{p}}{2\e^{s_c}}\big)-\lambda\big(\xt{x},\xt{k}+\frac{\xt{p}}{2\e^{s_c}}\big)\Big)^2\\
&\leq 2\frac{\lvert t_1-t_2\rvert}{\e^{s+\ga+\theta s_c}}\Big(\sup_{\xt{x},\xt{k}}\lvert \nabla_{\xt{k}}\lambda(\xt{x},\xt{k})\rvert ^2 \int_{\lvert \xt{p}\rvert<1} d\xt{p}\frac{1}{ \lvert \xt{p}\rvert ^{d+\theta-1}}+\sup_{\xt{x},\xt{k}}\lvert\lambda(\xt{x},\xt{k})\rvert ^2  \int_{\lvert \xt{p}\rvert>1} d\xt{p}\frac{1}{ \lvert \xt{p}\rvert ^{d+\theta}}  \Big),
\end{split}\end{equation}
and
\begin{equation}\label{diam}\begin{split}
diam^2_{M}([0,T]) & \leq 2\E\Big[\tilde{V}^2\Big(\frac{t_1}{\e^{s+\ga}}\Big)\Big] +2\E\Big[\tilde{V}^2\Big(\frac{t_2}{\e^{s+\ga}}\Big)\Big]  \\
&\leq \int m(d\xt{p}) \frac{2}{\mathfrak{g}^2(\xt{p})+\e^{2\ga}(\xt{k}\cdot\xt{p})^2}\Big(\lambda\big(\xt{x},\xt{k}-\frac{\xt{p}}{2\e^{s_c}}\big)-\lambda\big(\xt{x},\xt{k}+\frac{\xt{p}}{2\e^{s_c}}\big)\Big)^2\\
&\leq \frac{2}{\e^{(\theta+2\beta) s_c}}b^2(\xt{x},\xt{k}),
\end{split}\end{equation}
where
\[\begin{split}b^2(\xt{x},\xt{k})=& \, \int_{\lvert \xt{p}\rvert<1} d\xt{p}\frac{1}{ \lvert \xt{p}\rvert ^{d+\theta+2\beta-2}}
\int_{-1/2}^{1/2}du \lvert\nabla_\xt{x} \lambda(\xt{x},\xt{k}+u\xt{p})\rvert ^2\\
&+\int_{\lvert \xt{p}\rvert>1} d\xt{p}\frac{1}{ \lvert \xt{p}\rvert ^{d+\theta+2\beta}} \Big(\lambda\big(\xt{x},\xt{k}-\frac{\xt{p}}{2}\big)-\lambda\big(\xt{x},\xt{k}+\frac{\xt{p}}{2}\big)\Big)^2.
\end{split}\] 
Here, $diam_{M}([0,T])$ stands for the diameter of $[0,T]$ under the pseudo-metric $M$. Following the proof of \cite[Theorem 1.3.3 pp. 14]{adlertaylor}, and thanks to \eqref{metre} and \eqref{diam}, we have
\[\mathbb{E}\left[\sup_{t\in[0,T]}\Big\lvert \tilde{V}\Big(\frac{t}{\e^{s+\ga}}\Big) \Big\rvert^2 \right]\leq C_1\left(\int_0^{2b(\xt{x},\xt{k})/\e^{(\theta+2\beta) s_c/2}} \sqrt{\ln\left( C_2\frac{T}{r^2\e^{s+\ga+\theta s_c}}  \right)}dr\right)^2.\]
Consequently, 
\[\begin{split}
\e^{1+\ga}\E\big[\sup_{t\in[0,T]} \|\mathcal{L}_{1,\e}\lambda(t)\|^2_{L^2(\mathbb{R}^{2d})}\big]& \leq C  \int d\xt{x}d\xt{k} \Big(\int_0^{2\e^{(s+\ga-2\beta s_c)/2}b(\xt{x},\xt{k})} \sqrt{\ln\Big( C_2\frac{T}{r^2}  \Big)}dr\Big)^2\\
& \leq C \e^{(s+\ga-2\beta s_c)/2} \int d\xt{x}d\xt{k}b(\xt{x},\xt{k}) \int_0^1 \ln\Big( C_2\frac{T}{r^2}  \Big)dr  <+\infty,
\end{split}\]
which concludes the proof of Lemma \ref{bound3}, since $s_c=(1-s)/\theta$ and $s_c<s<(s+\ga)/(2\beta)$, and then the proof of Lemma \ref{bound2}.
$\square$
\end{proof}
$\square$
\end{proof}

\begin{proof}[of Lemma \ref{A1}]
The pseudogenerator $\mathcal{A}^\e \big(f^\e _1\big)(t)$ is given by 
\[\begin{split}
\mathcal{A}^\e \big(f^\e _1\big)(t)&=-\e^{(1-\ga)/2-s} f'(W_{\e,\lambda} (t))\big<W_\e(t),\mathcal{L}_\e\lambda(t)\big>_{L^2(\mathbb{R}^{2d})}\\
&+\e^{1-s}f'(W_{\e,\lambda} (t))\big<W_\e(t),\mathcal{L}_{\e}\big( \mathcal{L}_{1,\e}\lambda(t)\big)(t)\big>_{L^2(\mathbb{R}^{2d})}\\
 &+\e^{1-s}f''(W_{\e,\lambda} (t))\big<W_\e(t),\mathcal{L}_\e\lambda(t)\big>_{L^2(\mathbb{R}^{2d})}\big<W_\e(t),\mathcal{L}_{1,\e}\lambda(t)\big>_{L^2(\mathbb{R}^{2d})}\\
&+ \e^{(1+\ga)/2+s_c}f''(W_{\e,\lambda} (t))W_{\e,\lambda_1}(t) \big<W_\e(t),\mathcal{L}_{1,\e}\lambda(t)\big>_{L^2(\mathbb{R}^{2d})},
\end{split}\]
so that one can see that the corrector $f^\e _1$ remove the singular terms from $\mathcal{A}^\e \big(f^\e _0\big)(t)$ and we have the following result.
\begin{lem}\label{bound4}
We have
\[\sup_{t\leq T}\sup_{\e\in(0,1)}\e^{2(1-s)}\E\left[\|\mathcal{L}_{\e}\big( \mathcal{L}_{1,\e}\lambda(t)\big)(t)\|^2_{L^2(\mathbb{R}^{2d})} \right]<+\infty,\]
and
\[\sup_{t\leq T}\sup_{\e\in(0,1)}\e^{2(1-s)}\E\big[\|\mathcal{L}_{\e}\lambda(t)\|^2_{L^2(\mathbb{R}^{2d})} \times\|\mathcal{L}_{1,\e}\lambda(t)\|^2_{L^2(\mathbb{R}^{2d})}\big]<+\infty.\]
\end{lem}
Consequently, we obtain
\[\begin{split}
\mathcal{A}^\e(f^\e _0 +f^\e _1)(t)=&\e^{1-s} f'(W_{\e,\lambda}(t))\Big[W_{\e,\lambda_1}(t)+\big<W_\e(t),\mathcal{L}_{\e}\big(\mathcal{L}_{1,\e}\lambda(t)\big)(t)\big>_{L^2(\mathbb{R}^{2d})}\Big]\\
&+\e^{1-s}f''(W_{\e,\lambda}(t))\big<W_\e(t),\mathcal{L}_{1,\e}\lambda(t)\big>_{L^2(\mathbb{R}^{2d})}\big<W_\e(t),\mathcal{L}_{\e}\lambda(t)\big>_{L^2(\mathbb{R}^{2d})}\\
&+\mathcal{O}(\e^{(1+\ga)/2+s_c}),
\end{split}\] 
and $\sup_{\e,t}\E[\lvert \mathcal{A}(f^\e_0+f^\e_1)(t)\rvert^2]<+\infty$. That conclude the proof of Lemma \ref{A1} and then the proof of Proposition \ref{tightness}. 

\begin{proof}[of Lemma \ref{bound4}]

A short computation gives
\[\begin{split}
\mathcal{L}_{\e}\big( \mathcal{L}_{1,\e}&\lambda\big)\big(t,\xt{x},\xt{k})=\\
&-\frac{1}{(2\pi)^{2d}}\iint\widehat{V}\big(\frac{t}{\e^{s+\ga}},d\xt{p}_1\big)\widehat{V}\big(\frac{t}{\e^{s+\ga}},d\xt{p}_2\big)e^{i(\xt{p}_1+\xt{p}_2)\cdot \xt{x}/\e^s}\\
&\hspace{0.5cm}\times\Big(\frac{1}{\mathfrak{g}(\xt{p}_2)-i\e^{\ga}\big(\xt{k}-\frac{\xt{p}_1}{2\e^{s_c}}\big)\cdot\xt{p}_2}\big(\lambda\big(\xt{x},\xt{k}-\frac{\xt{p}_1}{2\e^{s_c}}-\frac{\xt{p}_2}{2\e^{s_c}}\big)-\lambda\big(\xt{x},\xt{k}-\frac{\xt{p}_1}{2\e^{s_c}}+\frac{\xt{p}_2}{2\e^{s_c}}\big) \big)\\
&\hspace{1cm}-\frac{1}{\mathfrak{g}(\xt{p}_2)-i\e^{\ga}\big(\xt{k}+\frac{\xt{p}_1}{2\e^{s_c}}\big)\cdot\xt{p}_2}\big(\lambda\big(\xt{x},\xt{k}+\frac{\xt{p}_1}{2\e^{s_c}}-\frac{\xt{p}_2}{2\e^{s_c}}\big)-\lambda\big(\xt{x},\xt{k}+\frac{\xt{p}_1}{2\e^{s_c}}+\frac{\xt{p}_2}{2\e^{s_c}}\big) \big)\Big).
\end{split}\]
Moreover, the forth order moment of our Gaussian field $\widehat{V}$ is given by   
\[
\begin{split}
\E\big[\widehat{V}(t_1,d\xt{p}_1)\widehat{V}(t_2,d\xt{p}_2)&\widehat{V}^\ast(t_3,d\xt{p}_3)\widehat{V}^\ast(t_4,d\xt{p}_4)\big]=\\
&\hspace{0.4cm}(2\pi)^{2d}\tilde{R}(t_1-t_2,\xt{p}_1)\tilde{R}(t_3-t_4,\xt{p}_3)\delta(\xt{p}_1+\xt{p}_2)\delta(\xt{p}_3+\xt{p}_4)\\
&+(2\pi)^{2d}\tilde{R}(t_1-t_3,\xt{p}_1)\tilde{R}(t_2-t_4,\xt{p}_3)\delta(\xt{p}_1-\xt{p}_3)\delta(\xt{p}_2-\xt{p}_4)\\
&+(2\pi)^{2d}\tilde{R}(t_1-t_4,\xt{p}_1)\tilde{R}(t_2-t_3,\xt{p}_3)\delta(\xt{p}_1-\xt{p}_4)\delta(\xt{p}_2-\xt{p}_3),
\end{split}
\]
so that, using the smoothness of $\lambda$, \eqref{SDCAreg} and the change of variable $\xt{p}'=\xt{p}/\e^{s_c}$, we obtain
\[\begin{split}
\e^{2(1-s)}\E\big[&\|\mathcal{L}_{\e}\big( \mathcal{L}_{1,\e}\lambda\big)(t)\|^2_{L^2(\mathbb{R}^{2d})} \big]\leq\e^{2(1-s-\theta s_c)}\\
& \times C \Big[\Big(\int_{\lvert \xt{p}\rvert<1} d\xt{p}\frac{1}{ \lvert \xt{p}\rvert ^{d+\theta-1}}\Big)^2+ \Big(\int_{\lvert \xt{p}\rvert>1} d\xt{p}\frac{1}{ \lvert \xt{p}\rvert ^{d+\theta}}\Big)^2\Big]\Big(\|\nabla_{\xt{x}}\lambda\|^2_{L^2(\mathbb{R}^{2d})}+\|\lambda\|^2_{L^2(\mathbb{R}^{2d})}\Big)<+\infty,
\end{split}\]
since $s_c=(1-s)/\theta$. In the same way we have
\[\sup_{t\leq T}\sup_{\e\in(0,1)}\e^{2(1-s)}\E\big[\|\mathcal{L}_{\e}\lambda(t)\|^2_{L^2(\mathbb{R}^{2d})} \times\|\mathcal{L}_{1,\e}\lambda(t)\|^2_{L^2(\mathbb{R}^{2d})}\big]<+\infty.\]
$\square$
\end{proof}
$\square$ 
$\blacksquare$

\end{proof}

\end{proof}

\subsection{Identification of all subsequence limits}\label{identificationsec}

In this section, we identify all the converging subsequence limits of the process $(W_\e)_\e$, given by its tightness, as solutions of a deterministic diffusion equation. For the sake of simplicity we still denote by $(W_\e)_\e$ such a subsequence. Let us note that in this case all the limit processes are therefore deterministic, and the convergence of the process $(W_\e)_\e$ also holds in probability. We will see that this fractional diffusion equation is well posed. In particular, this will imply the convergence of the process $(W_\e)_\e$ itself to the unique solution of this diffusion equation.

\begin{prop}\label{identification}
Let $W$ be an accumulation point of $(W_\e)_\e$. Then, $W$ is the unique strong solution of the diffusion equation
 \begin{equation}\label{diffeq}
\partial_t W = -\sigma(\theta)(-\Delta_\xt{k})^{\theta/2}W,
\end{equation}
with $W(0,\xt{x},\xt{k})=W_0(\xt{x},\xt{k})$ defined by \eqref{w01}, and
\[\sigma(\theta)=\frac{2a(0) \theta\Gamma(1-\theta)}{(2\pi)^d}\int_{\mathbb{S}^{d-1}}dS(\xt{u})\lvert\xt{e}_1\cdot \xt{u} \rvert^\theta.\]
\end{prop}

To prove this proposition we use the notion of pseudogenerator introduced in Section \ref{pseudogene} and the perturbed-test-function technique that we have already used in Section \ref{tightnesssec} for the proof of the tightness. Thanks to the pseudogenerator we are able to characterize the accumulation points of $(W_\e)_\e$ as solutions of a well-posed martingale problem, but as we will see it will turn out that the martingale problem is only a deterministic  partial differential equation. However, as we saw in Section \ref{tightnesssec} the pseudogenerator $\mathcal{A}^\e(f^\e _0)$ has singular terms, so that we have to modify the test-function $f^\e _0$ with a small corrector $f^\e _1$ in order to remove these singular terms. As a result, $\mathcal{A}^\e(f^\e _0 +f^\e _1)$ has no more singular term but it is not yet a generator of a martingale problem. Consequently, we will introduce a second small corrector $f^\e _2$ so that  $\mathcal{A}^\e(f^\e _0 +f^\e _1+f^\e _2)$ is now a generator allowing us to characterize the accumulation points. 

\begin{proof}[of Proposition \ref{identification}]
In this proof we use the following notation
\[\varphi\otimes\psi(\xt{x}_1,\xt{k}_1,\xt{x}_2,\xt{k}_2)= \varphi(\xt{x}_1,\xt{k}_1)\psi(\xt{x}_2,\xt{k}_2).\]
 Let us introduce the second corrector 
\[\begin{split}
f^\e _2 (t)=&\e^{1-s} f'(W_{\e,\lambda} (t))\big<W_\e(t),H_{1,\e}(t)\big>_{L^2(\mathbb{R}^{2d})}\\
&+\e^{1-s} f''(W_{\e,\lambda} (t))\big<W_\e(t)\otimes W_\e(t),H_{2,\e}(t)\big>_{L^2(\mathbb{R}^{4d})},
\end{split}
\]
where
\[\begin{split}
H_{1,\e}(t,&\xt{x},\xt{k})=\\
&\frac{1}{(2\pi)^{2d}i^2}\int_{t}^{+\infty}du\iint \Big(\mathbb{E}^\e _t\big[ \widehat{V}(u/\e^{s+\ga},d\xt{p}_1) \widehat{V}(u/\e^{s+\ga},d\xt{p}_2)\big] -\mathbb{E}\big[ \widehat{V}(0,d\xt{p}_1) \widehat{V}(0,d\xt{p}_2)\big] \Big) \\
&\times  e^{i(\xt{p}_1+\xt{p}_2)\cdot\xt{x}/\e^s}e^{i(u-t)(\xt{p}_1+\xt{p}_2)\cdot \xt{k}/\e^s}\\
&\times\Big[\frac{1}{\mathfrak{g}(\xt{\xt{p}_2})-i \e^\ga \big(\xt{k}-\frac{\xt{p}_1}{2\e^{s_c}}\big)\cdot\xt{p}_2}\Big(\lambda\big(\xt{x},\xt{k}-\frac{\xt{p}_1}{2\e^{s_c}}-\frac{\xt{p}_2}{2\e^{s_c}}\big)-\lambda\big(\xt{x},\xt{k}-\frac{\xt{p}_1}{2\e^{s_c}}+\frac{\xt{p}_2}{2\e^{s_c}}\big)\Big)\\
&\hspace{2cm}-\frac{1}{\mathfrak{g}(\xt{\xt{p}_2})-i\e^\ga\big(\xt{k}+\frac{\xt{p}_1}{2\e^{s_c}}\big)\cdot\xt{p}_2}\Big(\lambda\big(\xt{x},\xt{k}+\frac{\xt{p}_1}{2\e^{s_c}}-\frac{\xt{p}_2}{2\e^{s_c}}\big)-\lambda\big(\xt{x},\xt{k}+\frac{\xt{p}_1}{2\e^{s_c}}+\frac{\xt{p}_2}{2\e^{s_c}}\big)\Big)\Big] ,
\end{split}\]
and
\[\begin{split}
H_{2,\e}(t,&\xt{x}_1,\xt{k}_1,\xt{x}_2,\xt{k}_2)=\\
&\frac{1}{(2\pi)^{2d}i^2}\int_{t}^{+\infty}du\iint \Big(\mathbb{E}^\e _t\big[ \widehat{V}(u/\e^{s+\ga},d\xt{p}_1) \widehat{V}(u/\e^{s+\ga},d\xt{p}_2)\big] -\mathbb{E}\big[ \widehat{V}(0,d\xt{p}_1) \widehat{V}(0,d\xt{p}_2)\big] \Big) \\
&\times e^{i\xt{p}_1\cdot\xt{x}_1/\e^s}e^{i\xt{p}_2\cdot\xt{x}_2/\e^s}  e^{i(u-t)(\xt{p}_1\cdot \xt{k}_1+\xt{p}_2\cdot \xt{k}_2)/\e^s}\frac{1}{\mathfrak{g}(\xt{\xt{p}_1})-i\e^\ga\xt{k}_1\cdot\xt{p}_1}\\
&\times \Big(\lambda\big(\xt{x}_1,\xt{k}_1-\frac{\xt{p}_1}{2\e^{s_c}}\big)-\lambda\big(\xt{x}_1,\xt{k}_1+\frac{\xt{p}_1}{2\e^{s_c}}\big)\Big)\\
&\times\Big(\lambda\big(\xt{x}_2,\xt{k}_2-\frac{\xt{p}_2}{2\e^{s_c}}\big)-\lambda\big(\xt{x}_2,\xt{k}_2+\frac{\xt{p}_2}{2\e^{s_c}}\big)\Big).
\end{split}\]
Let 
\[f^\e (t)=f^\e _0 (t) + f^\e _1 (t)+ f^\e _2 (t).\] 
According to Theorem \ref{martingale}, $\big(M^\e _{f^\e} (t) \big)_{t\geq 0}$ is an $( \mathcal{F}^\e _t)$-martingale. That is, for every bounded continuous function $\Phi$, every sequence $0< s_1<\cdots <s_n \leq s <t$, and every family $(\mu_j)_{j\in \{1,\dots,n\}}\in L^2(\mathbb{R}^{2d})^n$, we have
\begin{equation}\label{MGe}\mathbb{E}\Big[ \Phi\big(W _{\e,\mu_j}(s_j),1\leq j \leq n\big)\Big( f^\e (t) - f^\e (s)-\int _s ^t  \mathcal{A}^\e f^\e (u)du \Big) \Big]=0,\end{equation}
where after a long but straightforward computation using  \eqref{markovvar}, we have
\begin{equation}\label{A2}\begin{split}
\mathcal{A}^\e(f^\e _0 +f^\e _1+f^\e_2)(t)=&\e^{1-s}f'(W_{\e,\lambda}(t))\big<W_\e(t),G_{1,\e}\lambda\big>_{L^2(\mathbb{R}^{2d})}\\
&+\e^{1-s} f''(W_{\e,\lambda}(t))\big<W_\e(t)\otimes W_\e(t),G_{2,\e}\lambda \big>_{L^2(\mathbb{R}^{4d})}\\
&+o(1),
\end{split}\end{equation}
with
\begin{equation}\label{G1e}\begin{split}
G_{1,\e}\lambda(\xt{x},\xt{k})=&-\frac{1}{(2\pi)^d}\int d\xt{p}\widehat{R}_0(\xt{p})\Big[\frac{1}{\mathfrak{g}(\xt{\xt{p}})-i\e^\ga\big(\xt{k}-\frac{\xt{p}}{2\e^{s_c}}\big)\cdot\xt{p}}\Big(\lambda(\xt{x},\xt{k})-\lambda\big(\xt{x},\xt{k}-\frac{\xt{p}}{\e^{s_c}}\big)\Big)\\
&\hspace{2cm}-\frac{1}{\mathfrak{g}(\xt{\xt{p}})-i\e^\ga\big(\xt{k}+\frac{\xt{p}}{2\e^{s_c}}\big)\cdot\xt{p}}\Big(\lambda\big(\xt{x},\xt{k}+\frac{\xt{p}}{\e^{s_c}}\big)-\lambda(\xt{x},\xt{k})\Big)\Big],
\end{split}\end{equation}
and
\begin{equation}\label{G2e}\begin{split}
G_{2,\e}\lambda(\xt{x}_1,\xt{k}_1,\xt{x}_2,\xt{k}_2)=&-\frac{1}{(2\pi)^d}\int d\xt{p} \frac{\mathfrak{g}(\xt{p})\widehat{R}_0(\xt{p})e^{i\xt{p}\cdot(\xt{x}_1-\xt{x}_2)/\e^s}}{(\mathfrak{g}(\xt{p})-i\e^\ga\xt{k}_1\cdot\xt{p})(2\mathfrak{g}(\xt{p})-i\e^\ga(\xt{k}_1-\xt{k}_2)\cdot\xt{p} )}\\
&\times \Big(\lambda\big(\xt{x}_1,\xt{k}_1-\frac{\xt{p}}{2\e^{s_c}}\big)-\lambda\big(\xt{x}_1,\xt{k}_1+\frac{\xt{p}}{2\e^{s_c}}\big)\Big)\\
&\times\Big(\lambda\big(\xt{x}_2,\xt{k}_2-\frac{\xt{p}}{2\e^{s_c}}\big)-\lambda\big(\xt{x}_2,\xt{k}_2+\frac{\xt{p}}{2\e^{s_c}}\big)\Big).
\end{split}\end{equation}
Let us remark that the second corrector $f^\e_2$ is effectively small.
\begin{lem}\label{bound6}
\[ \lim_\e \sup_{t\geq 0} \E[\lvert f^\e_2(t)\rvert]=0.\]
\end{lem}
As a result, using \eqref{A2}, Lemma \ref{bound2}, Lemma \ref{bound6}, we have the following result
\begin{lem}\label{MGlimit}
\[M_{f,\lambda}(t)=f (W_\lambda(t)) - f(W_\lambda(0))-\int _0^t  \partial_v f(W_{\lambda}(u))\big<W(u),G_{1}\lambda\big>_{L^2(\mathbb{R}^{2d})}du\]
is a martingale where 
\[
G_{1}\lambda(\xt{x},\xt{k})=-\sigma(\theta)(-\Delta)^{\theta/2}\lambda(\xt{x},\xt{k}).
\]
More particularly, let us consider $f$ be a smooth function so that $f(v)=v$, for all $ v$ such that $\lvert v\rvert\leq r\|\lambda\|_{L^2(\mathbb{R}^{2d})}$, then
\[M_{\lambda}(t)=W_\lambda(t) - W_\lambda(0)-\int _0 ^t  \big<W(u),G_{1}\lambda\big>_{L^2(\mathbb{R}^{2d})} du\]
is a martingale with a quadratic variation equal to $0$. Consequently, $M_{\lambda}=0$, that is $W$ is a deterministic weak solution of the diffusion equation \eqref{diffeq}.
\end{lem}
To show the weak uniqueness of this equation, let us assume that $W_0=0$. Moreover, it is clear that this diffusion equation with initial condition $\lambda_0\in L^2(\mathbb{R}^d)$ admits a unique strong  solution that we denote by $\lambda^\theta$ of the form \eqref{formula}. As result, for all $T>0$
\[
\big<W(T),\lambda_0\big>_{L^2(\mathbb{R}^{2d})}=\big<W(T),\tilde{\lambda}^\theta(T)\big>_{L^2(\mathbb{R}^{2d})}=\int_0^T \big<W(t), \partial_t \tilde{\lambda}^\theta(t)+\sigma(\theta)(-\Delta)^{\theta/2} \tilde{\lambda}^\theta(t)\big>_{L^2(\mathbb{R}^{2d})} dt=0,
\]
for $ \tilde{\lambda}^\theta(t)= \lambda^\theta(T-t)$. As a result, all the accumulation points are also strong solutions of the diffusion equation \eqref{diffeq}. That concludes the proof a Proposition \ref{identification}.

\begin{proof}[of Lemma \ref{bound6}]
According to \eqref{markovvar}
\[\begin{split}
\E^\e_t\Big[\widehat{V}\Big(u+&\frac{t}{\e^{s+\ga}},d\xt{p}_1\Big)\widehat{V}\Big(u+\frac{t}{\e^{s+\ga}},d\xt{p}_2\Big)\Big]-\E\Big[\widehat{V}(0,d\xt{p}_1)\widehat{V}(0,d\xt{p}_2)\Big]\\
&=e^{-(\mathfrak{g}(\xt{p}_1)+\mathfrak{g}(\xt{p}_2))u}\widehat{V}\Big(\frac{t}{\e^{s+\ga}},d\xt{p}_1\Big)\widehat{V}\Big(\frac{t}{\e^{s+\ga}},d\xt{p}_2\Big)-(2\pi)^d e^{-2\mathfrak{g}(\xt{p}_1)u}\widehat{R}_0(\xt{p}_1)\delta(\xt{p}_1+\xt{p}_2).
\end{split}\]
Consequently,
\[
f^\e _2 (t)=\e^{1+\ga}\Big[ f'(W_{\e,\lambda} (t))\big<W_\e(t),\tilde{H}_{1,\e}(t)\big>_{L^2(\mathbb{R}^{2d})}+ f''(W_{\e,\lambda} (t))\big<W_\e(t)\otimes W_\e(t),\tilde{H}_{2,\e}(t)\big>_{L^2(\mathbb{R}^{4d})}\Big],
\]
where
\[\begin{split}
\tilde{H}_{1,\e}(t,\xt{x},&\xt{k})=\frac{1}{(2\pi)^{2d}i^2}\iint \widehat{V}(t/\e^{s+\ga},d\xt{p}_1) \widehat{V}(t/\e^{s+\ga},d\xt{p}_2)\frac{ e^{i(\xt{p}_1+\xt{p}_2)\cdot\xt{x}/\e^s}}{(\mathfrak{g}(\xt{p}_1)+\mathfrak{g}(\xt{p}_2)-i\e^\ga\xt{k}\cdot(\xt{p}_1+\xt{p}_2))}\\
&\times  \Big[\frac{1}{\mathfrak{g}(\xt{\xt{p}_2})-i\e^\ga\big(\xt{k}-\frac{\xt{p}_1}{2\e^{s_c}}\big)\cdot\xt{p}_2}\Big(\lambda\big(\xt{x},\xt{k}-\frac{\xt{p}_1}{2\e^{s_c}}-\frac{\xt{p}_2}{2\e^{s_c}}\big)-\lambda\big(\xt{x},\xt{k}-\frac{\xt{p}_1}{2\e^{s_c}}+\frac{\xt{p}_2}{2\e^{s_c}}\big)\Big)\\
&\hspace{1cm}-\frac{1}{\mathfrak{g}(\xt{\xt{p}_2})-i\e^\ga\big(\xt{k}+\frac{\xt{p}_1}{2\e^{s_c}}\big)\cdot\xt{p}_2}\Big(\lambda\big(\xt{x},\xt{k}+\frac{\xt{p}_1}{2\e^{s_c}}-\frac{\xt{p}_2}{2\e^{s_c}}\big)-\lambda\big(\xt{x},\xt{k}+\frac{\xt{p}_1}{2\e^{s_c}}+\frac{\xt{p}_2}{2\e^{s_c}}\big)\Big)\Big]\\
&-\frac{1}{(2\pi)^d}\int d\xt{p}\frac{\widehat{R}_0(\xt{p})}{2\mathfrak{g}(\xt{p})}\Big[\frac{1}{\mathfrak{g}(\xt{\xt{p}})+i\e^\ga \big(\xt{k}-\frac{\xt{p}}{2\e^{s_c}}\big)\cdot\xt{p}}\Big(\lambda\big(\xt{x},\xt{k}-\frac{\xt{p}}{\e^{s_c}}\big)-\lambda(\xt{x},\xt{k})\Big)\\
&\hspace{1cm}-\frac{1}{\mathfrak{g}(\xt{\xt{p}})+i\e^\ga\big(\xt{k}+\frac{\xt{p}}{2\e^{s_c}}\big)\cdot\xt{p}}\Big(\lambda(\xt{x},\xt{k})-\lambda\big(\xt{x},\xt{k}+\frac{\xt{p}}{\e^{s_c}})\Big)\Big],
\end{split}\]
and
\[\begin{split}
\tilde{H}_{2,\e}(t,\xt{x}_1,\xt{k}_1,\xt{x}_2,\xt{k}_2)=&\frac{1}{(2\pi)^{2d}i^2} \iint \widehat{V}(t/\e^{s+\ga},d\xt{p}_1) \widehat{V}(t/\e^{s+\ga},d\xt{p}_2) \\
&\times \frac{e^{i\xt{p}_1\cdot\xt{x}_1/\e^s}e^{i\xt{p}_2\cdot\xt{x}_2/\e^s}}{(\mathfrak{g}(\xt{\xt{p}_1})-i\e^\ga\xt{k}_1\cdot\xt{p}_1)(\mathfrak{g}(\xt{\xt{p}_1})+\mathfrak{g}(\xt{\xt{p}_2})-i\e^\ga(\xt{k}_1\cdot\xt{p}_1+\xt{k}_2\cdot\xt{p}_2))}\\
&\times \Big(\lambda\big(\xt{x}_1,\xt{k}_1-\frac{\xt{p}_1}{2\e^{s_c}}\big)-\lambda\big(\xt{x}_1,\xt{k}_1+\frac{\xt{p}_1}{2\e^{s_c}}\big)\Big)\\
&\times\Big(\lambda\big(\xt{x}_2,\xt{k}_2-\frac{\xt{p}_2}{2\e^{s_c}}\big)-\lambda\big(\xt{x}_2,\xt{k}_2+\frac{\xt{p}_2}{2\e^{s_c}}\big)\Big)\\
&-\frac{1}{(2\pi)^d}\int d\xt{p} \frac{\widehat{R}_0(\xt{p})e^{i\xt{p}\cdot(\xt{x}_1-\xt{x}_2)/\e^s}}{(\mathfrak{g}(\xt{p})-i\e^\ga\xt{k}_1\cdot\xt{p})(2\mathfrak{g}(\xt{p})-i\e^\ga(\xt{k}_1-\xt{k}_2)\cdot\xt{p} )}\\
&\times \Big(\lambda\big(\xt{x}_1,\xt{k}_1-\frac{\xt{p}}{2\e^{s_c}}\big)-\lambda\big(\xt{x}_1,\xt{k}_1+\frac{\xt{p}}{2\e^{s_c}}\big)\Big)\\
&\times\Big(\lambda\big(\xt{x}_2,\xt{k}_2-\frac{\xt{p}}{2\e^{s_c}}\big)-\lambda\big(\xt{x}_2,\xt{k}_2+\frac{\xt{p}}{2\e^{s_c}}\big)\Big).
\end{split}\]
In the same way as in Lemma \ref{bound4}, we have
\begin{equation}\label{boundlem}\begin{split}
\sup_{t}\E[\|\tilde{H}_{1,\e}(t)\|^2_{L^{\mathbb{R}^{2d}}}&+\|\tilde{H}_{2,\e}(t)\|^2_{L^{\mathbb{R}^{2d}}}]
\leq \e^{-2(\theta+2\beta)s_c} \\
&\times C \Big[\Big(\int_{\lvert \xt{p}\rvert<1} d\xt{p}\frac{1}{ \lvert \xt{p}\rvert ^{d+\theta-1} }\Big)^2+ \Big(\int_{\lvert \xt{p}\rvert>1} d\xt{p}\frac{1}{ \lvert \xt{p}\rvert ^{d+\theta}}\Big)^2\Big]\\
&\times \Big(\|D^2_{\xt{x}}\lambda\|^2_{L^2(\mathbb{R}^{2d})}+\|\nabla_{\xt{x}}\lambda\|^2_{L^2(\mathbb{R}^{2d})}+\|\lambda\|^2_{L^2(\mathbb{R}^{2d})}\Big),
\end{split}\end{equation}
and 
\[1+\ga-(\theta+2\beta)s_c>0 \Longleftrightarrow s>(1-\ga\theta/(2\beta))/(1+\theta/(2\beta))=1/(2\kappa_\ga), \]
that concludes the proof of Lemma \ref{bound6}.
$\square$
\end{proof}
\begin{proof}[of Lemma \ref{MGlimit}]
We want to pass to the limit $\e\to 0$ in \eqref{MGe}. Making the change of variable $\xt{p}'=\xt{p}/\e^{s_c}$ in \eqref{G1e} and \eqref{G2e}, the $\e$-power in \eqref{A2} becomes $1-s-\theta s_c$. Consequently, the particular choice $s_c=(1-s)/\theta$ is made to obtain a nontrivial limit of \eqref{A2} as $\e\to 0$. Consequently, with the change of variable $\xt{p}'=\xt{p}/\e^{s_c}$ in \eqref{G2e}, and using the Riemann-Lebesgue Lemma and the dominated convergence theorem, we have
\[\lim_{\e} \| \e^{1-s} G_{2,\e}\lambda\|^2 _{L^2(\mathbb{R}^{4d})}=0,\]
since $s_c=(1-s)/\theta<s$. Moreover, with the change of variable $\xt{p}'=\xt{p}/\e^{s_c}$ in \eqref{G1e} and the dominated convergence theorem, we have
\begin{equation}\label{bound8}\lim_{\e} \|  \e^{1-s}G_{1,\e}\lambda-G_1\lambda\|^2 _{L^2(\mathbb{R}^{2d})}=0,\end{equation}
where 
\begin{equation}\label{G1}\begin{split}
G_{1}\lambda(\xt{x},\xt{k})&=\frac{a(0)}{(2\pi)^d}\int d\xt{p}\frac{1}{\lvert \xt{p}\rvert^{d+\theta}}\Big(\lambda(\xt{x},\xt{k}+\xt{p})+\lambda(\xt{x},\xt{k}-\xt{p})-2\lambda(\xt{x},\xt{k})\Big)\\
&=\frac{2a(0)}{(2\pi)^d}\int d\xt{p}\frac{1}{\lvert \xt{p}\rvert^{d+\theta}}\Big(\lambda(\xt{x},\xt{k}+\xt{p})-\lambda(\xt{x},\xt{k})\Big)\\
&=-\sigma(\theta)(-\Delta)^{\theta/2}\lambda.
\end{split}\end{equation}
Consequently, we have
\[
\mathcal{A}^\e(f^\e _0 +f^\e _1+f^\e_2)(t)=f'(W_{\e,\lambda}(t))\big<W_\e(t),G_{1}\lambda\big>_{L^2(\mathbb{R}^{2d})}+o(1),
\]
that concludes the proof of Lemma \ref{MGlimit} using Lemma \ref{bound2}, Lemma \ref{bound6}.
$\square$
\end{proof}
$\blacksquare$
\end{proof}

\section{Proof of Theorem \ref{thasymptotic3}}\label{proof2}

The proof of this theorem is quite similar to the previous one. The proof of the tightness is exactly the same. However, in this theorem we have assumed $s_c=s$, but also either $\beta<1/2$ or $\ga>0$.  To characterize the accumulation points we use exactly the same perturbed test functions as in the proof of Theorem \ref{thasymptotic}, to obtain
\begin{equation}\label{A2bis}\begin{split}
\mathcal{A}^\e(f^\e _0 +f^\e _1+f^\e_2)(t)=&\e^{1-s}f'(W_{\e,\lambda}(t))\big<W_\e(t),G_{1,\e}\lambda\big>_{L^2(\mathbb{R}^{2d})}\\
&+\e^{1-s}f''(W_{\e,\lambda}(t))\big<W_\e(t)\otimes W_\e(t),G_{2,\e}\lambda \big>_{L^2(\mathbb{R}^{4d})}\\
&+o(1),
\end{split}\end{equation}
where $G_{1,\e}\lambda$ is defined by \eqref{G1e}, and $G_{2,\e}\lambda$ by \eqref{G2e} with $s=s_c$. However, we still have \eqref{bound8} for the drift term $f'$, but with the change of variable $\xt{p}'=\xt{p}/\e^{s}$ there are no fast phases anymore, and we obtain that
\[\begin{split}
M_{f,\lambda}(t)=f (W_\lambda(t)) - f(W_\lambda(0))-\int _0^t &  f'(W_{\lambda}(u))\big<W(u),G_{1}\lambda\big>_{L^2(\mathbb{R}^{2d})}\\
&+ f''(W_{\lambda}(u))\big<W(u)\otimes W(u),G_{2}(\lambda,\lambda)\big>_{L^2(\mathbb{R}^{4d})}du
\end{split}\]
is a martingale where $G_1$ is defined by \eqref{G1}, and
\begin{equation}\label{G2}\begin{split}
G_2(\lambda_1,\lambda_2)(\xt{x}_1,\xt{k}_1,\xt{x}_2,\xt{k}_2)=&-\frac{1}{(2\pi)^d}\int \frac{d\xt{p}}{\lvert \xt{p}\rvert ^{d+\theta}} e^{i\xt{p}\cdot(\xt{x}_1-\xt{x}_2)}  \\
&\times \Big(\lambda_1\big(\xt{x}_1,\xt{k}_1-\frac{\xt{p}}{2}\big)-\lambda_1\big(\xt{x}_1,\xt{k}_1+\frac{\xt{p}}{2}\big)\Big)\\
&\times\Big(\lambda_2\big(\xt{x}_2,\xt{k}_2-\frac{\xt{p}}{2}\big)-\lambda_2\big(\xt{x}_2,\xt{k}_2+\frac{\xt{p}}{2}\big)\Big).
\end{split}
\end{equation}
In this theorem the second order term $f''$ has not been killed by the Riemann-Lebesgue Lemma, so that the limiting point $W$ is not deterministic anymore. As a result we need to study the finite dimensional distributions
\[\lim_\e \E\Big[\prod_{j=1}^N \big<W_\e(t_j),\lambda_j\big>^{n_j}_{L^2(\mathbb{R}^{2d})}\Big]\]
to characterize all the accumulation points of $(W_\e)_\e$. To do that, we follow the technique used in \cite{fannjianguni,weinryb} and let us consider the tensorial process $W^M_\e (t)=\bigotimes_{j=1}^M W_\e (t)$ on $L^2(\mathbb{R}^{2dM})$. In the same way as the case $M=1$, we can show that $W^M_\e (t)$ is a tight process in $L^2(\mathbb{R}^{2dM})$ equipped with the weak topology and for all its accumulation points $W^M$ 
\[
M_{\lambda}(t)=W^M_\lambda(t) - W^M_\lambda(0)-\int _0^t  \big<W^M (u),G^M_{1}\lambda+\tilde{G}^M_{2}\lambda\big>_{L^2(\mathbb{R}^{2dM})}du
\]
is a martingale for all $\lambda \in L^2(\mathbb{R}^{2dM})$. Here, $G^M_1$ and $\tilde{G}^M_2$ are defined by
\[
G^M_1\lambda=-\sum_{j=1}^M \sigma(\theta) (-\Delta_{\xt{k_j}})^{\theta/2}\lambda
\] 
and 
\[
\begin{split}
\tilde{G}^M_2\lambda(&\xt{x}_1,\xt{k}_1,\cdots,\xt{x}_M,\xt{k}_M)\\
&=-\sum_{\substack{j_1,j_2=1\\j_1\not=j_2}}^M\frac{1}{(2\pi)^d}\int \frac{d\xt{p}}{\lvert \xt{p}\rvert ^{d+\theta}} e^{i\xt{p}\cdot(\xt{x}_{j_1}-\xt{x}_{j_2})}  \\
&\times \Big(\lambda\big(\xt{x}_1,\xt{k}_1,\cdots,\xt{x}_{j_1},\xt{k}_{j_1}-\frac{\xt{p}}{2},\cdots,\xt{x}_M,\xt{k}_M\big)-\lambda\big(\xt{x}_1,\xt{k}_1,\cdots,\xt{x}_{j_1},\xt{k}_{j_1}+\frac{\xt{p}}{2},\cdots,\xt{x}_M,\xt{k}_M\big)\Big)\\
&\times \Big(\lambda\big(\xt{x}_1,\xt{k}_1,\cdots,\xt{x}_{j_2},\xt{k}_{j_2}-\frac{\xt{p}}{2},\cdots,\xt{x}_M,\xt{k}_M\big)-\lambda\big(\xt{x}_1,\xt{k}_1,\cdots,\xt{x}_{j_2},\xt{k}_{j_2}+\frac{\xt{p}}{2},\cdots,\xt{x}_M,\xt{k}_M\big)\Big).
\end{split}
\]
As a result,  $\E[W^N]$ is a weak solution of the differential equation
\begin{equation}\label{equnicite}\partial_t \lambda^M=(G^M_1+\tilde{G}^M_2)\lambda^M.\end{equation}
Let $\lambda_0\in L^2(\mathbb{R}^{2dM})$ such that its Fourier transform with respect to $(\xt{k}_1,\cdots,\xt{k}_M)$, $\widehat{\lambda}^\xt{k}_0$, belongs to $\mathcal{C}^\infty_0(\mathbb{R}^{2dM})$. Solving \eqref{equnicite} in the Fourier domain, we show the existence and uniqueness  of a smooth function $\lambda^M$ of \eqref{equnicite} in the strong sense with initial condition $\lambda_0$. As result, if $\E[W^M(0)]=0$, for all $T>0$,
\[\begin{split}
\big<\E[W^M(T)],\lambda_0\big>_{L^2(\mathbb{R}^{2dM})}&=\big<\E[W^M(T)],\tilde{\lambda}^M(T)\big>_{L^2(\mathbb{R}^{2dM})}\\
&=\int_0^T \big<\E[W^M(t)], \partial_t \tilde{\lambda}^M(t)+\sigma(\theta)(-\Delta)^{\theta/2} \tilde{\lambda}^M(t)\big>_{L^2(\mathbb{R}^{2dM})} dt=0,
\end{split}\]
for $ \tilde{\lambda}^M(t)= \lambda^M(T-t)$. Consequently, by a density argument, we obtain the weak uniqueness of \eqref{equnicite}, and the moments
\[\E\Big[ W^M(t,\xt{x}_1,\xt{k}_1,\cdots,\xt{x}_M,\xt{k}_M)\Big]=\E\Big[\prod_{j=1}^M W(t,\xt{x}_j,\xt{k}_j)\Big]\]
are therefore uniquely determined for all accumulation point $W$. Let
\[\tilde{W}(t,\xt{x},\xt{k})=\frac{1}{(2\pi)^d}\int d\xt{q} \widehat{W}^\xt{k}_0(\xt{x},\xt{q})\exp\Big(i\xt{k}\cdot\xt{q}+i \int \mathcal{B}_t(d\xt{p})e^{i\xt{p}\cdot\xt{x}}( e^{-i\xt{q}\cdot \xt{p}/2} - e^{i\xt{q}\cdot \xt{p}/2} ) \Big).\]
Using the îto's formula and the weak uniqueness of \eqref{equnicite}, we obtain for all accumulation point $W^M$ of $(W^M_\e)_\e$
\[
\E\Big[ W^M(t,\xt{x}_1,\xt{k}_1,\cdots,\xt{x}_M,\xt{k}_M)\Big]=\E\Big[\prod_{j=1}^M \tilde{W}(t,\xt{x}_j,\xt{k}_j)\Big].
\]
Consequently, we have identified the one-dimensional finite distributions for all accumulation point $W$,
\[\lim_\e \E\Big[\prod_{j=1}^N \big<W_\e(t),\lambda_j\big>^{n_j}_{L^2(\mathbb{R}^{2d})}\Big]=\E\Big[\prod_{j=1}^N \big<W(t_j),\lambda_j\big>^{n_j}_{L^2(\mathbb{R}^{2d})}\Big]=\E\Big[\prod_{j=1}^N \big<\tilde{W}(t),\lambda_j\big>^{n_j}_{L^2(\mathbb{R}^{2d})}\Big].\]

To conclude the proof of Theorem \ref{thasymptotic3}, following a classical argument regarding the proof of uniqueness of martingale problems \cite[ Proposition 4.27 pp. 326]{Karatzas}: If the one-dimensional distributions of two solutions are the same, then their finite dimensional distributions are also the same. Consequently,
 \[\lim_\e \E\Big[\prod_{j=1}^N \big<W_\e(t_j),\lambda_j\big>^{n_j}_{L^2(\mathbb{R}^{2d})}\Big]=\E\Big[\prod_{j=1}^N \big<W(t_j),\lambda_j\big>^{n_j}_{L^2(\mathbb{R}^{2d})}\Big]=\E\Big[\prod_{j=1}^N \big<\tilde{W}(t_j),\lambda_j\big>^{n_j}_{L^2(\mathbb{R}^{2d})}\Big].\]

\end{document}